\documentclass[12pt]{article}
\usepackage{graphicx}\usepackage{amsmath}\usepackage{amsfonts}\usepackage{amssymb}\usepackage{color}\usepackage{float}
\providecommand{\U}[1]{\protect\rule{.1in}{.1in}}
\newtheorem{theorem}{Theorem}

\newtheorem{lemma}[theorem]{Lemma}
\newtheorem{proposition}[theorem]{Proposition}\newtheorem{remark}[theorem]{Remark}

\newenvironment{proof}[1][Proof]{\noindent\textbf{#1.} }{\ \rule{0.5em}{0.5em}}
\def\myblacksquare{\rule{1.2ex}{1.2ex}}
\def\bzero{\mathbf{0}}
\begin{document}
$\ $

\vspace{2.cm}
\begin{center}
{\Large \textbf{Critical configurations of solid bodies and}}

\vspace{.2cm}
{\Large \textbf{the Morse theory of MIN functions}}

\vspace{.6cm} {\large \textbf{Oleg Ogievetsky$^{1,2,3}$ and Senya Shlosman$^{1,4,5}$}}

\vskip .3cm $^{1}$Aix Marseille Universit\'{e}, Universit\'{e} de Toulon, CNRS, \newline CPT UMR 7332, 13288, Marseille, France

\vskip .05cm $^{2}$I.E.Tamm Department of Theoretical Physics, Lebedev Physical Institute, Leninsky prospect 53, 119991, Moscow, Russia

\vskip .05cm $^{3}$Kazan Federal University, Kremlevskaya 17, Kazan 420008, Russia

\vskip .05cm $^{4}$Inst. of the Information Transmission Problems, RAS, Moscow, Russia

\vskip .05cm $^{5}$Skolkovo Institute of Science and Technology, Moscow, Russia
\end{center}

\vspace{.2cm}
\begin{abstract}\noindent
We study the manifold of clusters of nonintersecting congruent solid bodies, all touching the central ball $B\subset\mathbb{R}^{3}$ of radius one. Two main examples are clusters of balls and clusters of infinite cylinders. We introduce the notion of \textit{critical cluster} and we study several critical clusters of balls and of cylinders. For the case of cylinders some of our critical clusters are new. We also establish the criticality properties of clusters, introduced earlier by W. Kuperberg.
\end{abstract}

\section{Introduction}
In this paper we will study manifolds $\mathcal{C}$ comprised by configurations of collections of solid bodies $\bar{\Lambda}_{1}
,\ldots,\bar{\Lambda}_{k}\subset\mathbb{R}^{3},$ $k>1$, touching the central unit ball $B\subset\mathbb{R}^{3}$. That is, the point $G$ of our manifold
$\mathcal{C}$ will be a configuration of non-intersecting solid bodies, $G=\left\{  \Lambda_{1},\ldots,\Lambda_{k}\right\}$, where each $\Lambda_{i}$ is
congruent to the corresponding shape $\bar{\Lambda}_{i}$ and is touching the unit ball $B$. It is allowed that some distances between bodies of $G$ are
zero. Any collection of such type we will call a \textit{cluster with a core} $B$ or just a \textit{cluster. }

\vskip .2cm
Evidently, the group $SO\left(  3\right)  $ acts on each $\mathcal{C}.$ So it is natural to study these manifolds $\operatorname{mod}\left(  SO\left(
3\right)  \right)  .$

\vskip .2cm
For the examples of various clusters see pictures below.

\vskip .2cm
By a deformation of the cluster $G$ we mean a continuous curve $G\left( t\right)  $ in the space of clusters, with $G\left(  0\right)  =G$. That means
that each $\Lambda_{i}$ touches the central ball $B$ during the process of the deformation.

\vskip .2cm
We call a cluster $G$ \textit{rigid, }if any deformation $G\left(  t\right)$ of $G$ has a form
\[ G\left(  t\right)  =g\left(  t\right)  G\ ,\]
where $g\left(  t\right)  \in SO\left(  3\right)  $ is a curve of rotations of $\mathbb{R}^{3},$ $g\left(  0\right)  =e.$ In other words, the only
deformation of $G$ available is the global rotation of $G$ as a solid body.

\vskip .2cm
We call a cluster $G$ \textit{flexible}, if it is not rigid, but during each deformation $G\left(  t\right)  $ some distances $\mathrm{dist}\left(
\Lambda_{i},\Lambda_{j}\right)  $ which were zero at $t=0$ remain zero at later moments $t$ (at least up to a moment $t_{0}>0$ which might depend on
the deformation $G\left(  \ast\right)  $).

\vskip .2cm
We say that $G$ can be unlocked, if there exists a continuous deformation $G\left(  t\right)  $ of $G$, such that for any $t>0$ all the distances
between the members $\Lambda_{i}$ in the cluster $G\left(  t\right)$ are positive (while each $\Lambda_{i}$ always touches the central ball during the move).

\vskip .2cm
An example of rigid cluster is the icosahedral cluster of balls, see section \ref{ico} below (note that the dodecahedral cluster of balls can be unlocked).

\vskip .2cm
An example of a flexible cluster is the arrangement of 5 balls of radius $1+\sqrt{2}$ around central unit ball $B$, one touching $B$ at the North pole,
one at the South pole, the other three touching along the equator. For another example see section \ref{four}.

\vskip .2cm
Finally, we call a cluster $G$ \textit{critical, }if for any smooth deformation $G\left(  t\right)  $ of $G$ all the distances $\mathrm{dist}
\left(  \Lambda_{i}\left(  t\right)  ,\Lambda_{j}\left(  t\right)  \right)$ between the solids $\Lambda_{i}$ which were zero at $t=0$ -- i.e.
$\mathrm{dist}\left(  \Lambda_{i}\left(  0\right)  ,\Lambda_{j}\left( 0\right)  \right)  =0$ -- obey the estimate
\begin{equation}
\mathrm{dist}\left(  \Lambda_{i}\left(  t\right)  ,\Lambda_{j}\left(
t\right)  \right)  =o\left(  t\right)  \text{ as }t\rightarrow 0\ .\label{crit}
\end{equation}
If a critical cluster $G$ can be unlocked, then it is called a \textit{saddle} cluster. Other critical clusters are called \textit{(local) maxima}, for
obvious reasons.

\vskip .2cm
In the present review we consider two types of solid bodies arrangements. One type consists of arrangements of balls $\Lambda_{i}$ of equal
radius $r,$ around $B$. Another type consists of arrangements of (infinite, right, circular) congruent cylinders around $B$. For balls, we present some
results of the paper \cite{KKLS}. For cylinders, we review the results of our recent studies \cite{OS1, OS2, OS3,OS4}.

\section{Critical clusters of balls}
\subsection{Maximal clusters of balls \label{ico}}
The icosahedral cluster $I_{12}$ of 12 equal balls gives an example of a maximal cluster. In 1943 Fejes-T\'{o}th has shown that

\vskip .2cm
~(1) The maximum radius of $12$ equal spheres touching a central sphere of radius $1$ is
\[ r_{max}(12) = \frac{1}{\sqrt{\frac{5+ \sqrt{5}}{2}} -1} \approx1.1085085\ . \]

\vskip .2cm
(2) An extremal cluster achieving this radius is formed by the balls centered at $12$ vertices of a regular icosahedron.

\vskip .2cm
One can call therefore the icosahedral cluster {\it the globally maximal}.

\vskip .2cm
There are other maximal clusters of $12$ equal spheres$.$ One of them, $A_{6,6},$ is given by equal balls centered at the vertices of a uniform
6-antiprism (note that the radii of these balls are less than $1$). In general, for every $n,$ the cluster of $2n$ equal spheres centered at the vertices of
a uniform $n$-antiprism, is locally maximal. For $n=2,3,4$ they are, in fact, global maxima.

\vskip .2cm
It is conjectured in \cite{KKLS} that for the case of $12$ balls there are other (sharp) locally maximal clusters of balls of radius $r$, $1<r<r_{max}
(12).$ The three candidate clusters are explicitly described there, and the proof requires just a computation, which, however, is too cumbersome.

\vskip .2cm
Both maximal clusters $I_{12}$ and $A_{6,6}$ are $PL$-maximal (piecewise-linear maximal). To explain this statement as well as to introduce
the connection with the Morse theory,  let $\mathcal{P}_{n}$ be the manifold of $n$-tuples of points on $\mathbb{S}^{2}$. To every cluster of $n$ balls we
associate a point in $\mathcal{P}_{n}$: it is the cluster of $n$ points at which the balls of the cluster are touching the central ball.
Consider the function $\delta$ on $\mathcal{P}_{n}$:
\[ \delta\left(  p\right)  =\min_{i\neq j}\mathrm{dist}\left(  x_{i},x_{j}\right)\ \ \text{where $p=\left\{  x_{1},...,x_{12}
\right\}  \in\mathcal{P}_{n}$} .\]
Define also the metric $\tilde{d}\left(  p^{\prime},p^{\prime\prime}\right)$ on $\mathcal{P}_{n}$ to be the Hausdorff distance between the two subsets of
$\mathbb{S}^{2}.$ To take into account the $SO\left(  3\right)  $-symmetry, which we want to factor out, we put also
\[ d\left(  p^{\prime},p^{\prime\prime}\right)  =\inf_{g\in SO\left(  3\right) }\tilde{d}\left(  p^{\prime},gp^{\prime\prime}\right) \ .\]
The criticality of the cluster is translated into the criticality of $\delta$ at the corresponding point: the point $p$ is called critical, if
\[ \delta\left(  p\right)  -\delta\left(  p^{\prime}\right)  \leq o\left( d\left(  p,p^{\prime}\right)  \right) \ .\]
The critical point $p_{I_{12}}\in\mathcal{P}_{12}$ -- i.e. the set of 12 vertices of the icosahedron $I_{12}$ -- is the point of the global maximum of
the function $\delta.$ The $PL$-maximality of $p_{I_{12}}$ is the property that for some constant $c_{I_{12}}>0$ we have
\begin{equation} \delta\left(  p_{I_{12}}\right)  -\delta\left(  p\right)  \geq c_{I_{12} }d\left(  p_{I_{12}},p\right)  ,\text{ provided }d\left(  p_{I_{12}},p\right)
\text{ is small enough.}\label{72}\end{equation}
In words, the condition $\left(
\ref{72}\right)  $ means that the function $\delta$ decays linearly with distance as we move away from $p_{I_{12}}.$ The same holds at the point
$p_{A_{6,6}},$ with a different constant $c_{A_{6,6}}>0$.
In fact, the relation $\left(  \ref{72}\right)  $ holds at the point $p_{I_{12}}$ without the smallness assumption;
at $p_{A_{6,6}}$ the relation $\left(  \ref{72}\right)  $ holds only locally.

\vskip .2cm
We conjecture that the same $PL$-maximality holds for any local maximum of the function $\delta$ on $\mathcal{P}_{n},$ i.e. for any locally maximal cluster
of $n$ equal balls, provided $n\geq 6.$

\vskip .2cm
Note that the maximal cluster $T_{3}$ of $3$ balls of radius $\frac{\sqrt{3} }{2-\sqrt{3}}$ touching the central unit ball is not $PL$-maximal; there
exists a curve $p\left(  t\right)  \subset\mathcal{P}_{3},$ $p\left( 0\right)  =p_{T_{3}},$ such that $d\left(  p_{T_{3}},p\left(  t\right)
\right)  =t,$ while the decay of $\delta$  is only quadratic:
\[ \delta\left(  p_{T_{3}}\right)  -\delta\left(  p\left(  t\right)  \right)\sim c_{T_{3}}t^{2}\ ,\ c_{T_{3}}>0\ . \]

\subsection{Saddle clusters of 12 balls}
The smallest value of $r$ for which there exists a critical cluster of 12 equal balls around the unit ball $B$ is $r=\frac{\sqrt{3}-1}{2\sqrt{2}-\sqrt{3}+1}\approx 0.3492$. The  saddle cluster of  balls of radius $r=\frac{\sqrt{3}-1}{2\sqrt{2}-\sqrt{3}+1}$ is a necklace of 12 balls
all touching $B$ at the equator, each touching two others.

\vskip .2cm
The most famous 12 ball clusters are the \text{FCC} (Face Centered Cubic) and the \text{HCP} (Hexagonal Closed Packed) clusters of unit balls. Each of them
can be part of a densest unit ball packing in $\mathbb{R}^{3}.$ The statement that the maximal density of a sphere packing in 3-dimensional space
is attained by the FCC packing, is called the Kepler Conjecture. It was proven by Hales and Ferguson, \cite{HF}.

\begin{figure}[H]
 \centering
\includegraphics[scale=0.28]{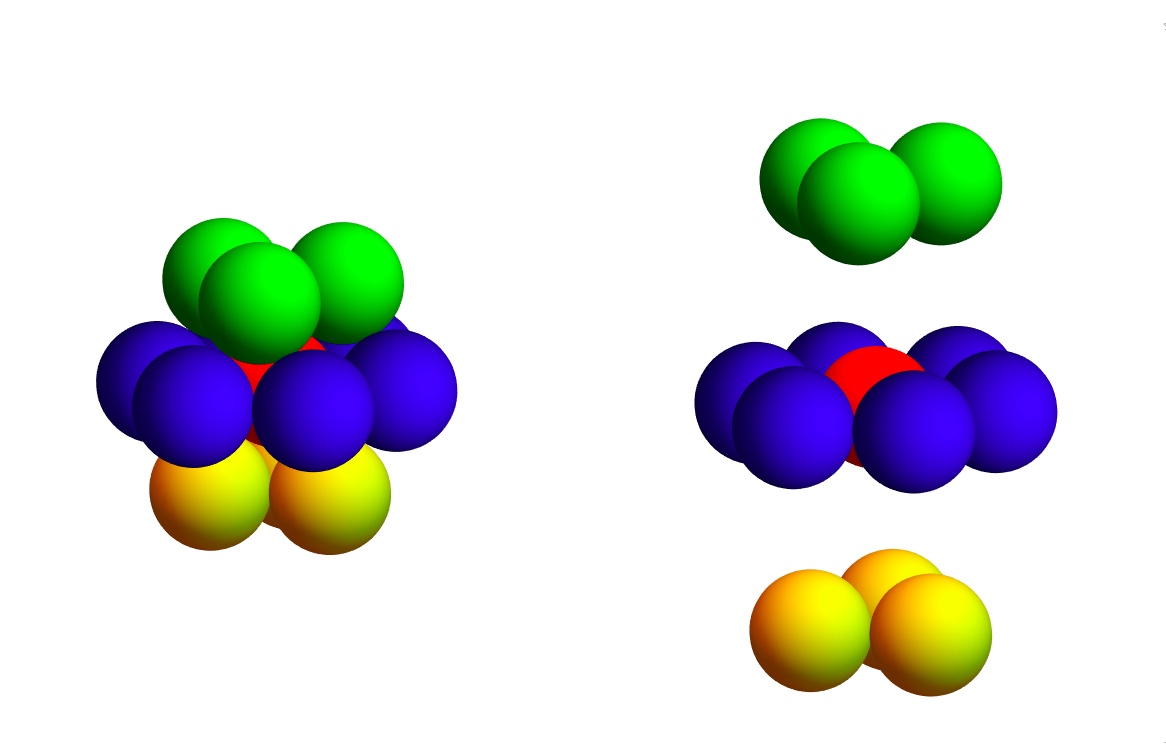}\caption{FCC cluster (left) and its layers (right)}
\label{FCC}
\end{figure}

\begin{figure}[H]
\centering
\includegraphics[scale=0.34]{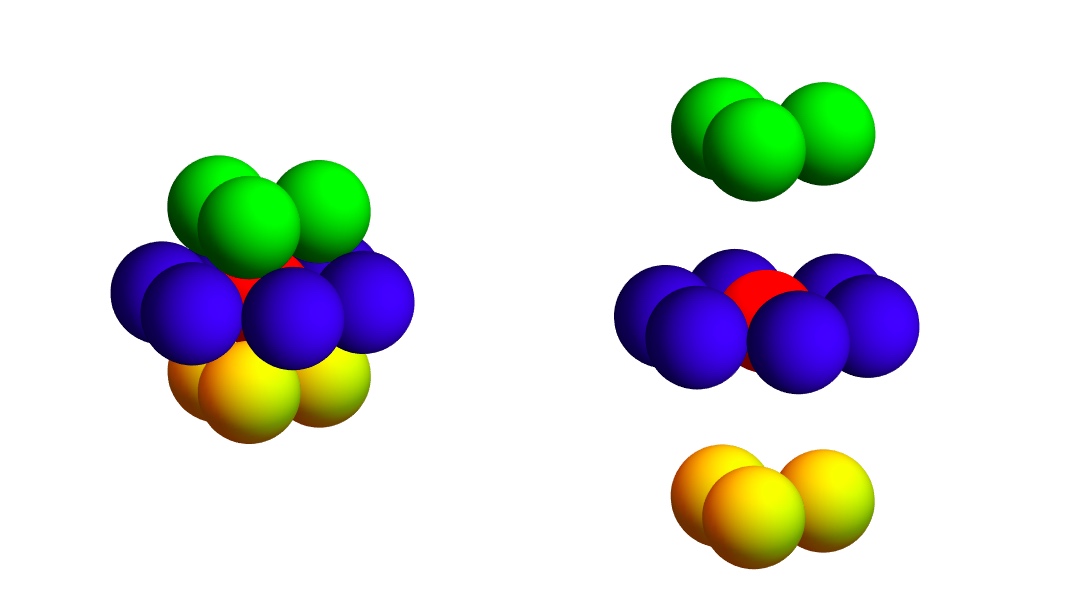}\caption{HCP cluster (left) and its layers (right)}
\label{HCP}
\end{figure}

The paper \cite{KKLS} contains detailed explanation of the fact that both FCC and HCP can be unlocked.

\vskip .2cm
The fact that the cluster FCC can be unlocked is mentioned in Chapter VII, \S $\,$2 in \cite{T} and is used, for example, in \cite{CS}, Appendix to Ch.
1. There it is built on the Coxeter constructions, see the book \cite{C}. The visualisation of this unlocking procedure can be also read out from the
movement of Buckminster Fuller's \textquotedblleft jitterbug\textquotedblright\ \cite{BF}, see the animation at https://www.youtube.com/watch?v=FfViCWntbDQ.
Note that the anima\-ted figure always has the symmetry group $\mathbb{A}_4$ although it is not immediately clear; this animation is related to the $\delta$-
rotation process discussed in Section \ref{delrotp} below.

\vskip .2cm
Chapter VII, \S $\,$2 of \cite{T} contains also a claim (without proof) that the HCP cluster is rigid: ``Dagegen ist die andere doppelwabenartige Anordnung
stabil". But in fact the HCP cluster can be unlocked as well, see \cite{KKLS} for a proof.

\vskip .2cm
As an indirect illustration of these claims one can consider the following triangulations of the FCC and HCP polyhedrons (the first one is taken from
\cite{C}). It is easy to see that both of them have the combinatorial type of the icosahedron.

\begin{figure}[H]
\centering
\vskip .4cm
\includegraphics[scale=0.22]{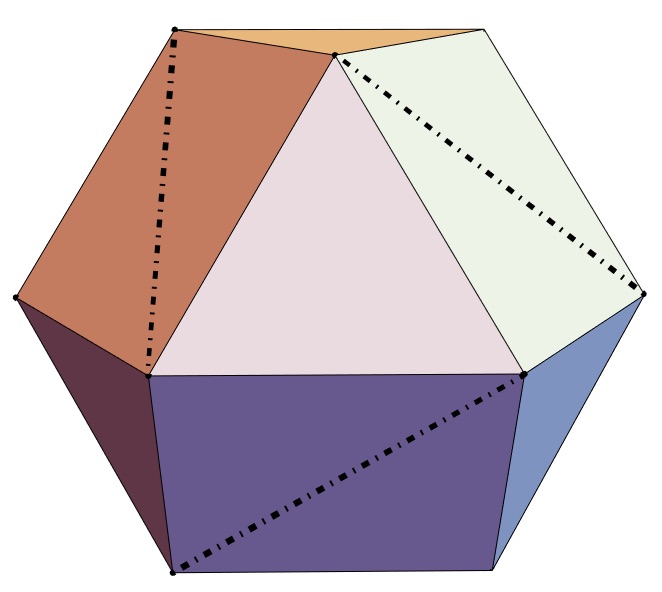}
\vskip .4cm
\caption{FCC triangulation}
\label{FCC-tri}
\end{figure}

\vspace{.6cm}
\begin{figure}[H]
\centering
\vskip .4cm
\includegraphics[scale=0.24]{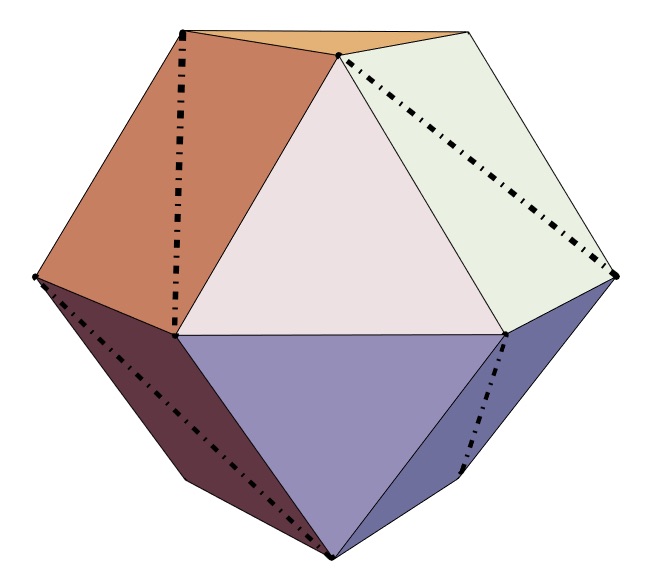}
\vskip .4cm
\caption{HCP triangulation}
\label{HCP-tri}
\end{figure}

It is interesting to note that if $G^{\prime}\left(  t\right)  ,G^{\prime \prime}\left(  t\right)  $ are two smooth unlocking deformations of FCC then the tangent vectors
to the paths $G^{\prime}\left(  t\right) ,G^{\prime\prime}\left(  t\right)  $ at $t=0$, i.e. at the point FCC coincide (up to a scalar factor); the same holds for HCP.
In words, that means that there is just one single vector, along which one can unlock the cluster FCC (and HCP). See \cite{KKLS} for details.

\vskip .2cm
The unlocking deformation for FCC can be chosen in such a way that 6 out of 12 balls do not move. For HCP the minimal number of fixed balls during the
unlocking deformation is 3.

\section{Critical clusters of cylinders}
We denote by $\mathcal{C}_{L}$ the space, modulo $SO(3)$, of clusters of $L$ infinite,
right, circular congruent cylinders, by $\mathcal{C}_{L}(r)$ the level subspace of $\mathcal{C}_{L}$ consisting of clusters of cylinders of radius $r$,
and by $\mathcal{C}_{L}(\geq r)$ the subspace of $\mathcal{C}_{L}$ consisting of clusters of cylinders of radius $\geq r$.

\subsection{Clusters of four cylinders \label{four}}

Consider the `manifold' $\mathcal{C}_{4}(r)$ of clusters of four cylinders of radius $r.$ At the value $r=1$ we find a critical cluster comprised by two
vertical cylinders interlaced with two horizontal cylinders. It can be visualized by removing two vertical cylinders on Figure \ref{octahedrConfCyl}.
It is easy to see that the resulting cluster can be unlocked, so it is a saddle.

\vskip .2cm
The level manifold $\mathcal{C}_{4}\left(  1+\sqrt{2}\right)$ (consisting of clusters of 4 cylinders of radius $r=1+\sqrt{2},$) contains a cluster of 4
parallel cylinders:

\vspace{.4cm}
\begin{figure}[H]
\centering
\includegraphics[scale=0.16]{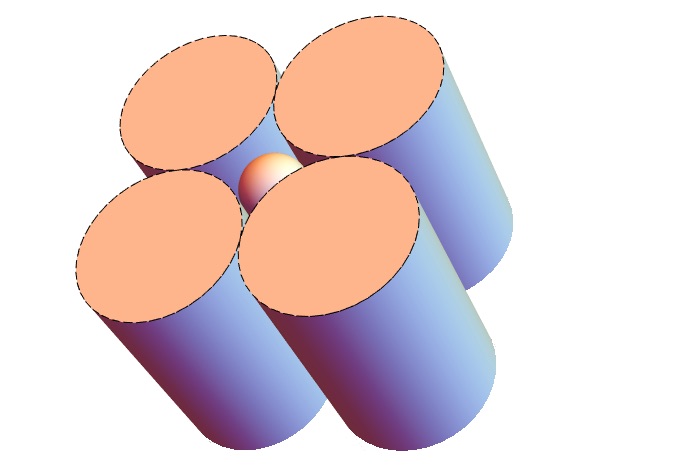}\caption{Initial
position}
\label{four parallel cylinders}
\end{figure}

This cluster is flexible:

\vspace{.4cm}
\begin{figure}[H]
\centering
\includegraphics[scale=0.226]{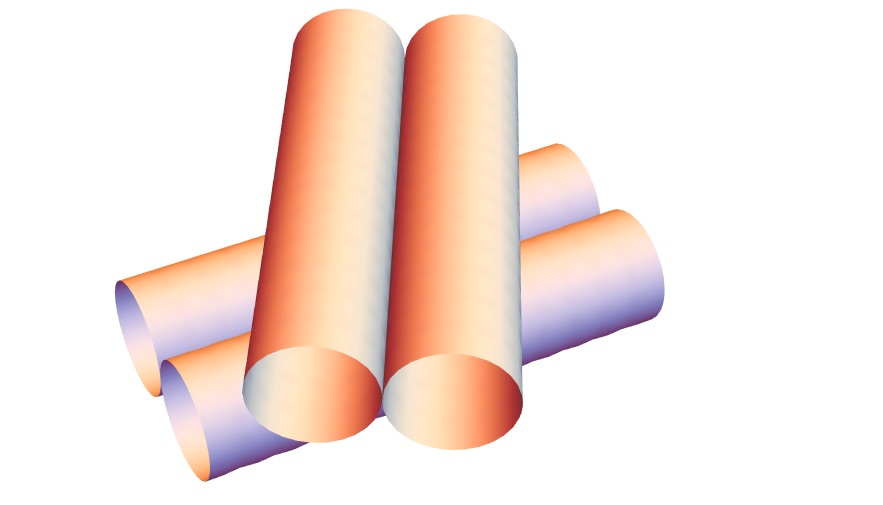}\caption{Motion of four
cylinders}
\label{four cylinders motion}
\end{figure}

We conjecture that the manifold $\mathcal{C}_{4}\left( 1+\sqrt{2}\right)$ has no other points, i.e. it is a circle $\operatorname{mod}\left(  SO\left(
3\right)  \right)  ,$ while the levels $\mathcal{C}_{4}\left(  r\right)$ with $r>1+\sqrt{2}$ are empty.

\subsection{Clusters of six cylinders}
The six cylinders case (and the corresponding `manifold' $\mathcal{C}_{6}$) turns out to be even more interesting.

\vskip .2cm
The question: - How many non-intersecting unit right circular cylinders of infinite length can touch a unit ball? \noindent- was asked by W.
Kuperberg, \cite{K}.

\vskip .3cm
Kuperberg presented several arrangements of clusters of 6 unit cylinders; it is difficult to imagine that there are clusters of 7 unit
cylinders, though no proof of this statement is known; see \cite{HS} for the proof that 8 unit non-intersecting cylinders of infinite length cannot touch
the unit ball.

\vskip .2cm
The two pictures below might suggest that the cluster $C_{6}$ of six parallel unit cylinders is maximal and flexible, and that the clusters of 6 cylinders
with $r>1$ do not exist:

\vspace{0.6cm}
\begin{figure}[H]
\centering
\includegraphics[scale=0.2]{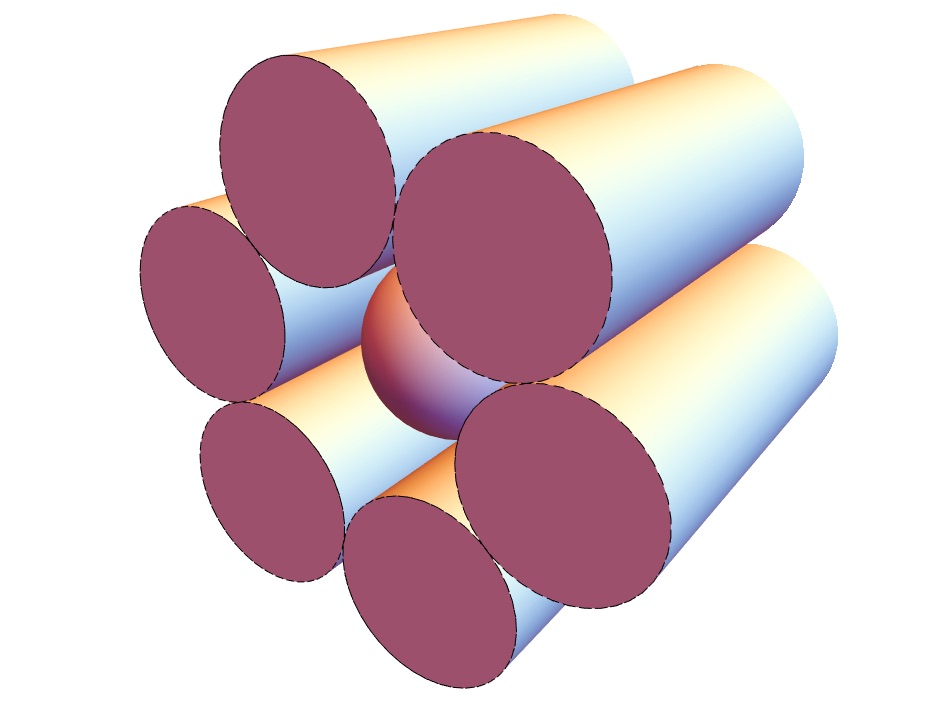}
\vskip .2cm
\caption{Cluster $C_{6}$}
\label{confC6}
\end{figure}

\vspace{0.2cm} \begin{figure}[H]
\centering
\includegraphics[scale=0.36]{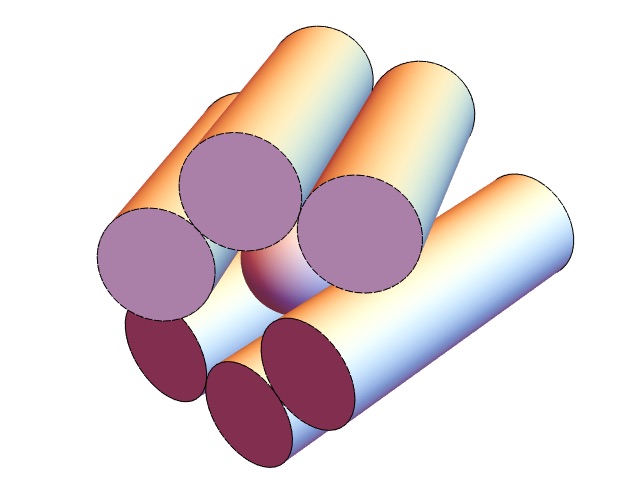}
\vskip .2cm
\caption{Non-rigidity
of $C_{6}$}
\label{nonrigconfC6}
\end{figure}

This, however, is not the case, and an example was presented by M. Firsching in his thesis, \cite{F}. In his example the radius $r$ of cylinders equals
$1.049659.$ This example was obtained by a numerical exploration of the corresponding 18-dimensional configuration manifold.

\vskip .2cm
One of the main results of our paper \cite{OS1} claims that in fact the cluster $C_{6}$ is critical;  -- in other words, it can be unlocked. To explain
this statement we first introduce notations, borrowed from \cite{OS1}.

\vskip .2cm
Let $\mathbb{S}^{2}\subset\mathbb{R}^{3}$ be the unit sphere, centered at the origin. For every $x\in\mathbb{S}^{2}$ by $TL_{x}$ we denote the set of all
(unoriented) tangent lines to $\mathbb{S}^{2}$ at $x.$ The manifold of tangent
lines to $\mathbb{S}^{2}$ we denote by $M$, and we represent a point in $M$ by
a pair $\left(  x,\tau\right)  $, where $\tau$ is a unit tangent vector to $\mathbb{S}^{2}$ at $x,$ though such a pair is not unique: the pair $\left(
x,-\tau\right)  $ is the same point in $M.$ We shall use the following coordinates on $M$. Let $\mathbf{x,y,z}$ be the standard coordinate axes in
$\mathbb{R}^{3}$. Let $R_{\mathbf{x}}^{\alpha},R_{\mathbf{y}}^{\alpha}$ and $R_{\mathbf{z}}^{\alpha}$ be the counterclockwise rotations about these axes
by an angle $\alpha$, viewed from the tips of axes. We call the point $\mathsf{N}=\left(  0,0,1\right)  $ the North pole, and $\mathsf{S}=\left(
0,0,-1\right)  $ -- the South pole. By \textit{meridians} we mean geodesics on $\mathbb{S}^{2}$ joining the North pole to the South pole. The meridian in the
plane $\mathbf{xz}$ with positive $\mathbf{x}$ coordinates will be called Greenwich. The angle $\varphi$ will denote the latitude on $\mathbb{S}^{2},$
$\varphi\in\left[  -\frac{\pi} {2},\frac{\pi}{2}\right]  ,$ and the angle $\varkappa\in\lbrack0,2\pi)$ -- the longitude, so that Greenwich corresponds
to $\varkappa=0.$ Every point $x\in\mathbb{S}^{2}$ can be written as $x=\left(  \varphi_{x},\varkappa_{x}\right)  .$ Finally, for each
$x\in\mathbb{S}^{2}$, we denote by $R_{x}^{\alpha}$ the rotation by the angle $\alpha$ about the axis joining $\left(  0,0,0\right)  $ to $x,$
counterclockwise if viewed from its tip, and by $\left(  x,\uparrow\right)$ we denote the pair $\left(  x,\tau_{x}\right)  ,$ $x\neq\mathsf{N,S,}$ where
the vector $\tau_{x}$ points to the North. We also abbreviate the notation $\left(  x,R_{x}^{\alpha} \uparrow\right)  $ to $\left(  x,\uparrow_{\alpha
}\right)  $.

\vskip .2cm
Let $u=\left(  x^{\prime},\tau^{\prime}\right)  ,$ $v=\left( x^{\prime\prime},\tau^{\prime\prime}\right)  $ be two lines in $M$. We denote
by $d_{uv}$ the distance between $u$ and $v$; clearly $d_{uv}=0$ iff $u\cap v\neq\varnothing.$ If the lines $u,v$ are not parallel then the square of
$d_{uv}$ is given by the formula
\[ d_{uv}^{2}=\frac{\det^{2}[\tau^{\prime},\tau^{\prime\prime},x^{\prime\prime
}-x^{\prime}]}{1-(\tau^{\prime},\tau^{\prime\prime})^{2}}\ ,\]
where $(\ast,\ast)$ is the scalar product. The cylinders $C_{u}\left( r\right)  $ and $C_{v}\left(  r\right)$, touching $\mathbb{S}^{2}$ at
$x^{\prime},x^{\prime\prime},$ having directions $\tau^{\prime},\tau^{\prime\prime},$ and radius $r,$ touch each other iff
\begin{equation} r=\frac{d_{uv}}{2-d_{uv}}\ . \label{11} \end{equation}
Indeed, when the cylinders touch each other, we have the proportion:
\begin{equation} \frac{d}{1}=\frac{2r}{1+r}\ . \label{dia}\end{equation}

We denote by $M^{6}$ the manifold of 6-tuples
\begin{equation} \mathbf{m}=\left\{  u_{1},...,u_{6}:u_{i}\in M,i=1,...,6\right\}  . \label{25} \end{equation}
We will study the function $D$ on $M^{6}$:
\begin{equation} D\left(  \mathbf{m}\right)  =\min_{1\leq i<j\leq 6}d_{u_{i}u_{j}}\ . \label{41}\end{equation}
We are especially interested in knowing its maximum, since it defines, via $\left(  \ref{11}\right)  ,$ the maximum radius of 6 non-intersecting equal
cylinders touching the unit ball.

\vskip .2cm
The generators of the cylinders in $C_{6}$ touching the ball define a point in $M^{6}$, shown on Figure \ref{confC6tan}. We denote it by
the same symbol $C_{6}$. Note that $D\left(  C_{6}\right)  =1$. \vspace{1cm}

\begin{figure}[H]
\vspace{-.6cm} \centering
\includegraphics[scale=0.2]{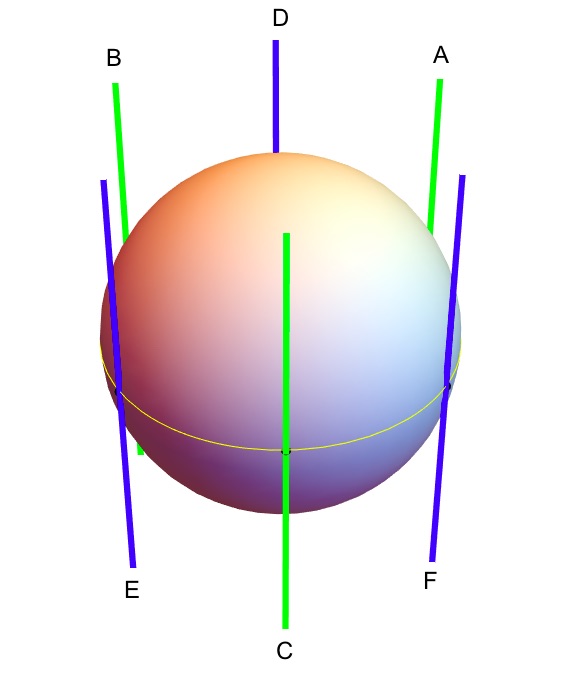}\caption{Cluster
$C_{6}$ of tangent lines}
\label{confC6tan}
\end{figure}

Now we are in the position to describe the `good' clusters $\mathbf{m}$ with high values of the function $D\left(  \mathbf{m}\right)  .$ We obtain them by
deforming the cluster $C_{6},$ which in our notation can be written as
\[ \begin{array}[c]{ll}
C_{6} & \equiv C_{6}\left(  0,0,0\right)  =\left\{  \left[  \left( 0,\frac{\pi}{6}\right)  ,\uparrow\right]  ,\left[  \left(  0,\frac{\pi}
{2}\right)  ,\uparrow\right]  ,\left[  \left(  0,\frac{5\pi}{6}\right) ,\uparrow\right]  ,\right. \\[0.8em]
& \hspace{3cm}\left.  \left[  \left(  0,\frac{7\pi}{6}\right)  ,\uparrow \right]  ,\left[  \left(  0,\frac{3\pi}{2}\right)  ,\uparrow\right]  ,\left[
\left(  0,\frac{11\pi}{6}\right)  ,\uparrow\right]  \right\}  .
\end{array}\]

Namely, we will explore the 6-tuples $C_{6}\left(  \varphi,\delta,\varkappa\right)  $, of the form
\begin{equation}\label{defconf}
\begin{array}[c]{ll}
& C_{6}\left(  \varphi,\delta,\varkappa\right)  = \left\{  A=\left[  \left( \varphi,\frac{\pi}{6}-\varkappa\right)  ,\uparrow_{\delta}\right]  ,D=\left[
\left(  -\varphi,\frac{\pi}{2} +\varkappa\right)  ,\uparrow_{\delta}\right],\right. \\[.8em]
& \hspace{2.74cm} B=\left[  \left(  \varphi,\frac{5\pi}{6}-\varkappa\right) ,\uparrow_{\delta}\right]  ,E=\left[  \left(  -\varphi,\frac{7\pi}
{6}+\varkappa\right)  ,\uparrow_{\delta}\right]  ,\\[.8em]
& \hspace{2.74cm} \left.  C=\left[  \left(  \varphi,\frac{3\pi}{2} -\varkappa\right)  ,\uparrow_{\delta}\right]  ,F=\left[  \left(
-\varphi,\frac{11\pi} {6}+\varkappa\right)  ,\uparrow_{\delta}\right]\right\}  .\end{array}\end{equation}
In words, the three points $\left[  \left(  0,\frac{\pi}{6}\right),\uparrow\right]  ,\left[  \left(  0,\frac{5\pi}{6}\right)  ,\uparrow\right]
$ and $\left[  \left(  0,\frac{3\pi}{2}\right)  ,\uparrow\right]  $ go upward by $\varphi,$ then `horizontally' by $-\varkappa,$ and then the three vectors
$\uparrow$ are rotated by $\delta,$ while the three remaining points go downward by $\varphi$, then `horizontally' by $\varkappa,$ and, finally, the
three vectors $\uparrow$ are rotated by $\delta$.

\vskip .2cm
For all $\varphi,\delta,\varkappa$ these clusters possess $\mathbb{D}_{3}\equiv\mathbb{Z}_{3}\times\mathbb{Z}_{2}$ symmetry. The group
$\mathbb{D}_{3}$ is generated by the rotations $R_{\mathbf{z}}^{120^{\circ}}$ and $R_{\mathbf{x}}^{180^{\circ}}.$ We denote by $\mathcal{C}^{3}\in M^{6}$
the 3-dimensional submanifold formed by 6-tuples $\left(  \ref{defconf} \right)  $.

\vskip .2cm
We claim that there exists a curve $\gamma$ in the manifold $\mathcal{C}^{3}$,
\begin{equation} \gamma(\varphi)=C_{6}\bigl(\varphi,\delta\left(  \varphi\right)
,\varkappa\left(  \varphi\right)  \bigr)\ ,\ \varphi\in\left[  0;\frac{\pi}
{2}\right]  \ , \label{curvegamma} \end{equation}
which starts at $C_{6}\left(  0,0,0\right)  $ for $\varphi=0$,
\begin{equation} \gamma(0)=C_{6}\left(  0,0,0\right)  \ , \label{curvegamma2} \end{equation}
such that the function $D\bigl(\gamma(\varphi)\bigr)$ is unimodal on $\gamma$, with maximum value $\sqrt{\frac{12}{11}}$, which corresponds to the value
\begin{equation} r_{\mathfrak{m}}=\frac{1}{8}\left(  3+\sqrt{33}\right)  \approx1.093070331\ .\label{30}\end{equation}
of the radii of the touching cylinders. This is summarized by the main result of our paper \cite{OS1}.

\begin{theorem}
\label{Main} The cluster $C_{6}\left(  0,0,0\right)  $ can be unlocked. Moreover,

\vskip.2cm \textbf{i. }There is a continuous curve $\gamma$, see (\ref{curvegamma}) and (\ref{curvegamma2}), on which the function
$D\bigl(\gamma(\varphi)\bigr)$ increases for $\varphi\in\left[  0,\varphi_{\mathfrak{m}}\right]  $ and decreases for $\varphi>\varphi_{\mathfrak{m}},$
with $\varphi_{\mathfrak{m}}=\arcsin\sqrt{\frac{3}{11}}.$ The explicit description of $\gamma$ is given by the relations (\ref{trajphi0})-(\ref{trajkappa}) below.

\vskip .3cm \textbf{ii.} At the point $\varphi_{\mathfrak{m}},\delta_{\mathfrak{m}} =\delta\left(  \varphi_{\mathfrak{m}}\right)  ,\varkappa
_{\mathfrak{m} }=\varkappa\left(  \varphi_{\mathfrak{m}}\right)$ we have
\[ D\Bigl(  C_{6}\left(  \varphi_{\mathfrak{m}},\delta_{\mathfrak{m}}
,\varkappa_{\mathfrak{m}}\right)  \Bigr)  =\sqrt{\frac{12}{11}}\ ,\]
so the radius of the corresponding cylinders is equal to
\[  r_{\mathfrak{m}}=\frac{1}{8}\left(  3+\sqrt{33}\right)  .\]

\end{theorem}

\vskip .2cm
The record cluster $C_{\mathfrak{m}}=C_{6}\left(  \varphi_{\mathfrak{m}},\delta_{\mathfrak{m}},\varkappa_{\mathfrak{m}}\right)  $ is
shown on Figure \ref{record1} below.

\vskip .2cm
There is an animation, on the page of Yoav Kallus \cite{Ka}, demonstrating the motion of the cluster of 6 cylinders along the curve $\gamma(\varphi)$.

\begin{figure}[H]
\centering
\includegraphics[scale=0.22]{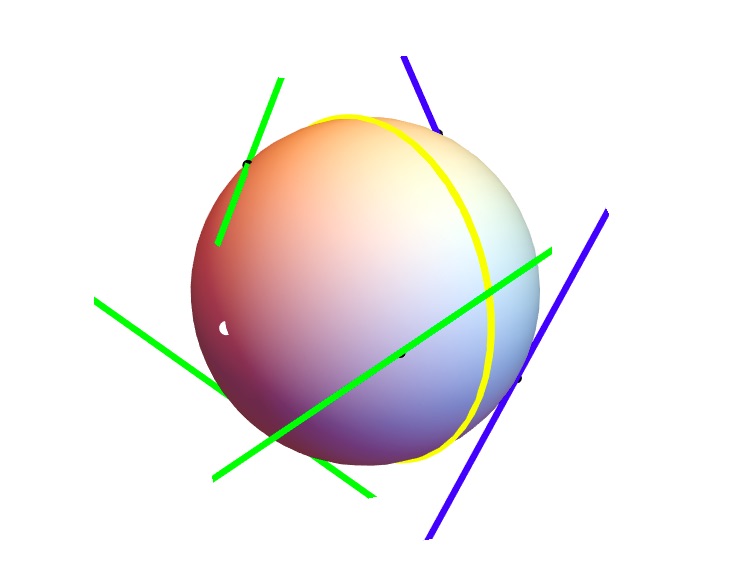}
\caption{Record cluster, the equator is yellow, the north pole is white}
\label{record1}
\end{figure}

\begin{figure}[H]
\centering
$\ \ \ $
\raisebox{.3cm}{\includegraphics[scale=0.71]{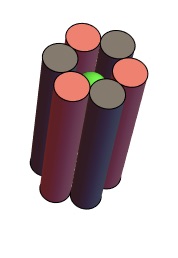}}$\ \ $
\includegraphics[scale=0.352]{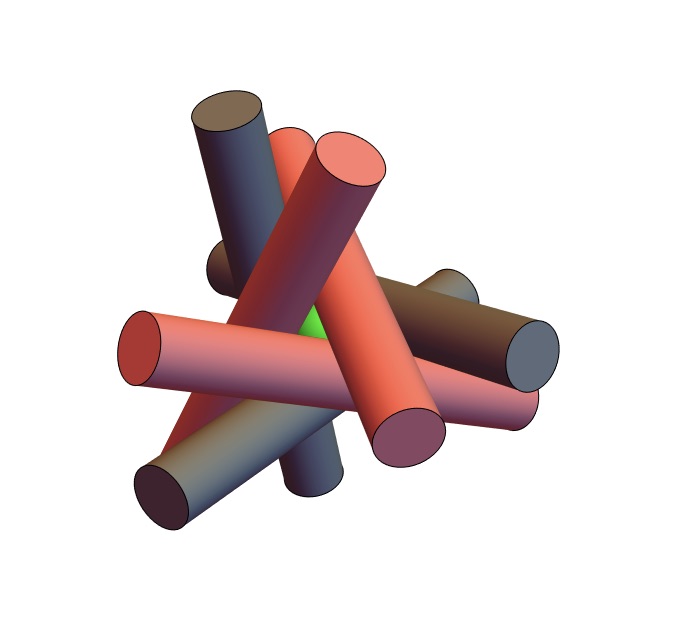} \parbox{11.8cm}{
\caption{Two clusters of cylinders: the cluster $C_6$ of six parallel cylinders of radius 1 (on the left) and the cluster $C_{\mathfrak{m}}$ of six cylinders of radius $\,\approx\! 1.0931$ (on the right) } }\end{figure}

\subsection{Rigidity of the cluster $C_{\mathfrak{m}}$}
The main result of the paper \cite{OS2} claims that the cluster $C_{\mathfrak{m}}$ is rigid. In other words, any small perturbation (apart of the global
rotation) of the line cluster shown on Figure \ref{record1} results in a smaller value of the function $D,$ defined in $\left(  \ref{41}\right)  .$

\begin{theorem}
\label{theorC6} The cluster $C_{\mathfrak{m}}\ $is a point of a sharp local maximum of the function $D$: for any point $\mathbf{m}$ in a vicinity of
$C_{\mathfrak{m}}$ we have
\[ D\left(  \mathbf{m}\right)  <\sqrt{\frac{12}{11}}=D\left(  C_{\mathfrak{m}}\right)\  .\]

\end{theorem}

\begin{remark}
There exists a 4-dimensional subspace $L_{quadr}$ in the tangent space of $M^{6}$ at $C_{\mathfrak{m}},$ such that for any $l\in L_{quadr}$ we have
\[ -c_{u}\left\Vert l\right\Vert t^{2}\leq D\left(  C_{\mathfrak{m}}+tl\right) -D\left(  C_{\mathfrak{m}}\right)  \leq-c_{d}\left\Vert l\right\Vert t^{2}\]
for $t$ small enough. Here $c_{d}$ and $c_{u}$ are some constants, $0<c_{d}\leq c_{u}<+\infty$ and $C_{\mathfrak{m}}+tl\in M^{6}$ stands for the
exponential map applied to the tangent vector $tl$.

For each tangent vector $l$ outside $L_{quadr}$ we have
\[ -c_{u}^{\prime}\left(  l\right)  t\leq D\left(  C_{\mathfrak{m}}+tl\right) -D\left(  C_{\mathfrak{m}}\right)  \leq-c_{d}^{\prime}\left(  l\right)  t\]
for $t$ small enough, where now $c_{d}^{\prime}\left(  l\right)  $ and $c_{u}^{\prime}\left(  l\right)  $ are some positive valued functions of $l$,
$0<c^{\prime}\left(  l\right)  \leq c^{\prime\prime}\left(  l\right) <+\infty$.
\end{remark}

\vskip.2cm \textbf{Note} \textbf{i). }The Remark above does not imply the maximality claim of our Theorem, as the following example shows:

\vskip .2cm
Let $f$ be a function of two variables defined by
\[ f:=\min\{u_{1},u_{2}\}\ \ \text{where}\ \ u_{1}=-y+3x^{2},u_{2}=y-x^{2}\ .\]
The function $f$ equals 0 at the origin. Consider an arbitrary ray $l$ starting at the origin. Clearly, for some time this ray evades the `horns' --
the region between the parabolas $y=3x^{2}$ and $y=x^{2}.$ But outside the horns the function $f$ is negative. Indeed, inside the the narrow parabola
$y=3x^{2}$ we have $u_{1}<0,u_{2}>0$ so $f$ is negative there; outside the wide parabola $y=x^{2}$ we have $u_{1}>0,u_{2}<0$ so $f$ is negative there as
well. Therefore the origin is a local maximum of $f$ restricted to $l,$ for any $l.$ Yet the origin is not a local maximum of the function $f$ on the
plane, because inside the horns the functions $u_{1}$ and $u_{2}$ are positive so $f$ there is positive as well.

\vskip .2cm
\textbf{Note} \textbf{ii). }The cluster $C_{\mathfrak{m}}$ is not centrally symmetric, so its image $-C_{\mathfrak{m}}$ under central symmetry
produces a different point of the manifold $\mathcal{C}_{6}.$ Hence, the last theorem implies that the submanifold $\mathcal{C}_{6}\left(  \geq d\right)
\subset\mathcal{C}_{6}$ of clusters $\mathbf{m}$ with $D\left(  \mathbf{m}
\right)  \geq d$ has at least two connected components for $d>\sqrt{\frac{12}{11}}-\varepsilon$ once $\varepsilon>0$ is small enough. We believe that these two
connected components stay disjoint for all $d>1$; more precisely, the rigid clusters $C_{\mathfrak{m}}$ and $-C_{\mathfrak{m}}$ can communicate only via
the saddle point cluster $C_{6}.$ Notwithstanding, the submanifold $\mathcal{C}_{6}\left(  \geq 1\right)  \subset\mathcal{C}_{6}$ (and also
$\mathcal{C}_{6}\left( \geq 1-\varepsilon\right)  \subset\mathcal{C}_{6}$ with small $\varepsilon$) is still not connected, as the next section shows.

\subsection{Galois symmetry\vspace{.1cm}}\label{hiddenGsymmetry}
While proving Theorem \ref{theorC6} we revealed a hidden symmetry of the formulas for the coefficients of the Taylor expansions of distances
between the tangent lines at points of the curve $\gamma$. Here we shortly describe this symmetry.

\vskip .2cm
The curve $\gamma$ admits the following parameterization:
\begin{equation} \sin\bigl(\varphi(x)\bigr)=2\sqrt{\frac{(1-x)x(1+x)}{1+7x+4x^{2}}}\ , \label{trajphi0}\end{equation}
\begin{equation} \tan\bigl(\delta(x)\bigr)=\sqrt{\frac{(1-x)(1+3x)}{x(1+7x+4x^{2})}}\ , \label{trajdelta0}\end{equation}
and
\begin{equation} \tan\bigl(\varkappa(x)\bigr)=\frac{x-1}{\sqrt{(1+x)(1+3x)}}\ . \label{trajkappa} \end{equation}
where $x$ ranges from 1 to 0. The record cluster $C_{\mathfrak{m}}$ corresponds to the value $x=1/2$.

\vskip .2cm
We reserve the same letters $\{A,B,C,D,E,F\}$, see Figure \ref{confC6tan}, for the tangent lines of the cluster
$C_{6}\bigl(\varphi\left(x\right),\delta\left(x\right),\varkappa\left(  x\right)  \bigr)$. At each value of $x$ the cluster
$C_{6}\bigl(\varphi\left(x\right),\delta\left(x\right),\varkappa\left(  x\right)  \bigr)$ has the symmetry group
$\mathbb{D}_3$ generated by the permutations $(A,B,C)(D,E,F)$ and $(A,D)(B,F)(C,E)$. The group $\mathbb{D}_6$, under which the initial cluster
$C_6$ is invariant, has the additional generator $\varsigma=(B,C)(D,F)$.

\vskip .2cm
The perturbed position of a line $J\in \{A,B,C,D,E,F\}$
in the cluster $C_{6}\bigl(\varphi\left(x\right),\delta\left(x\right),\varkappa\left(  x\right)  \bigr)$ is
$$J=J\bigl( \varkappa(x)+\vartheta_\varkappa\cdot J_{\varkappa }\, ,\,\varphi(x)+\vartheta_\varphi\cdot J_{\varphi }\, ,\,
\delta(x)+\vartheta_\delta\cdot J_{\delta }\bigr)\ ,$$
where $J_{\varkappa }$, $J_{\varphi }$ and $J_{\delta }$ are the perturbation parameters. We introduced the normalization constants $\vartheta_\varkappa, \vartheta_\varphi$ and $\vartheta_\delta$ needed to formulate the result.

\vskip .2cm
Let
$$ \mathfrak{p}_x =\sqrt{\frac{(1 + x) (1 + 3 x)}{3}}\ .$$

\begin{proposition}\label{geogasy} Let $x$ be a rational number between 0 and 1 such that $\mathfrak{p}_x $ is not rational.

\vskip .2cm
\noindent {\rm (i)} There exists a choice of the normalization constants $\vartheta_\varkappa, \vartheta_\varphi$ and $\vartheta_\delta$ such that
the Taylor coefficients  of the squares of distances belong to $\mathbb{Q}[\mathfrak{p}_x ]$.

\vskip .2cm
\noindent {\rm (ii)} The permutation $\varsigma$, composed with the Galois conjugation $\mathfrak{p}_x \to -\mathfrak{p}_x $
of $\mathbb{Q}[\mathfrak{p}_x ]$, restores the $\mathbb{D}_6$ symmetry of the cluster $C_{6}\bigl(\varphi\left(x\right),\delta\left(x\right),\varkappa\left(  x\right)  \bigr)$.
\end{proposition}
\noindent{\bf Note.} Let, for example, the angle $\delta$ be varied. The required normalization factor is
$$ \vartheta_\delta =\sqrt{\frac{1+x}{3x(1-x)(1+7x+4x^2)}}\ .$$

\subsection{Rigidity of the cluster $O_{6}$}
In \cite{K} W. Kuperberg suggested another cluster of six unit non-intersecting cylinders touching the unit sphere and asked whether it can
be unlocked. It is the cluster $O_{6}$ shown on Figure \ref{octahedrConfCyl}.

\begin{figure}[th]
\vspace{-.2cm}
\centering
\includegraphics[scale=0.7]{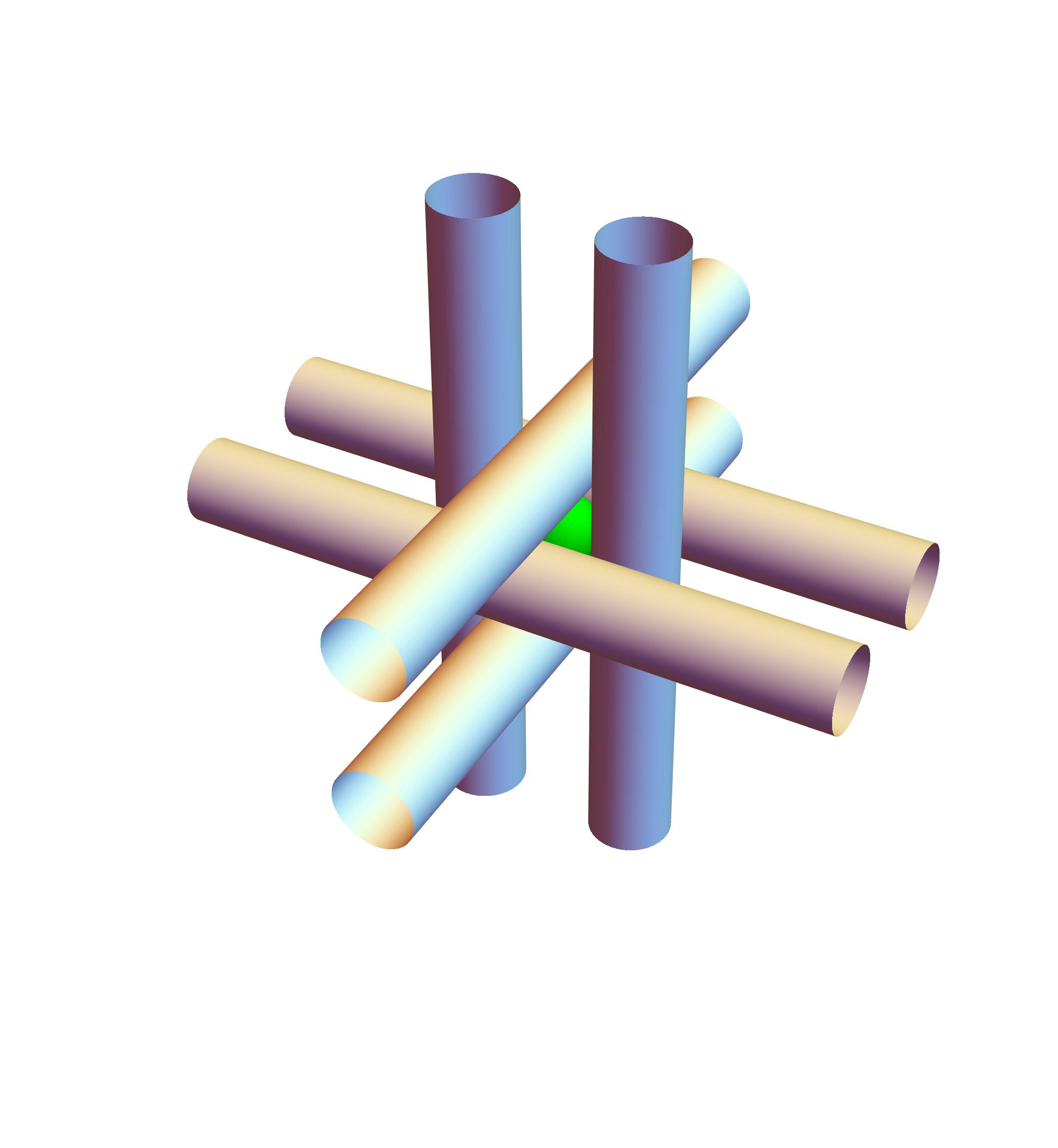}
\vspace{-1.2cm}
\caption{Cluster $O_{6}$ of
cylinders}
\label{octahedrConfCyl}
\end{figure}

\vskip.2cm The main result of our paper \cite{OS3} claims that this cluster cannot be unlocked.

\begin{theorem}
\label{octa} The cluster $O_{6}\ $is a point of local maximum of the function $D$: for any point $\mathbf{m}$ in the vicinity of $O_{6}$ we have
\[ D\left(  \mathbf{m}\right)  <1=D\left(  O_{6}\right) \ .\]

\end{theorem}

\begin{remark}
There exists a 6D subspace $L_{quadr}$ in the tangent space of $M^{6}$ at $O_{6},$ such that for any $l\in L_{quadr},$ $\left\Vert l\right\Vert =1$ we have
\[ -c^{\prime\prime}\left(  l\right)  t^{2}\leq D\left(  O_{6}+tl\right) -D\left(  O_{6}\right)  \leq-c^{\prime}\left(  l\right)  t^{2}\]
for $t$ small enough, where $O_{6}+tl\in M^{6}$ stands for the exponential map applied to the tangent vector $tl,$ and $0<c^{\prime}\left(  l\right)  \leq
c^{\prime\prime}\left(  l\right)  <+\infty.$ For each tangent vector $l$ outside $L_{quadr}$ we similarly have
\[ -c^{\prime\prime}\left(  l\right)  t\leq D\left(  O_{6}+tl\right)  -D\left( O_{6}\right)  \leq-c^{\prime}\left(  l\right)  t \]
for $t$ small enough, with $0<c^{\prime}\left(  l\right)  \leq c^{\prime\prime}\left(  l\right)  <+\infty.$\end{remark}

\textbf{Note: }the Note i) after the Theorem \ref{theorC6} applies here as well.

\section{Towards the theory of the critical points and critical values of the MIN functions}
The proofs of the above theorems boil down to the study of the `critical' points of the function $D$ on the manifold $\mathcal{C}_{6}.$ The difficulty
here lies in the fact that the function $D$, being a minimum of several analytic functions, is not smooth -- so is not at all a Morse
function. We do not have a complete version of the theory needed here. Rather, we present few results which cover a small part of a general picture.

\vskip .2cm
Let $F_{1}\left(  x\right)  ,\dots,F_{m}\left(  x\right)  $ be analytic functions in a neighborhood of $\bzero\in\mathbb{R}^{n}$
such that $F_{u}(\bzero)=0$, $u=1,\dots,m$, and let
\begin{equation} \mathsf{F}\left(  x\right)  :=\min\left\{  F_{1}\left(  x\right)  ,\dots ,F_{m}\left(  x\right)  \right\}  . \label{deffuF} \end{equation}
Let us consider the differentials $l_{u}$ and second differentials $q_{u}$ of the functions $F_{u}$ at $\bzero\in\mathbb{R}^{n}$ :
\begin{equation} F_{u}(x)=l_{u}\left(  x\right)  +q_{u}\left(  x\right)  +o(2)\ .\label{decofuf}\end{equation}

Here $l_{u}$-s and $q_{u}$-s are linear and, respectively, quadratic forms on the tangent space $T_{\bzero}\mathbb{R}^{n}$ and
$o(2)$ stands for higher order terms.

\vskip .2cm
We call the function
\[ \Delta\left(  x\right)  :=\min\left\{  l_{1}\left(  x\right)  ,\dots ,l_{m}\left(  x\right)  \right\} \]
the $PL$-differential of $\mathsf{F}.$ The range of the differential $\Delta$ can be either a whole line $\mathbb{R}^{1},$ or the negative half-line. In the second
case we say that $\bzero\in\mathbb{R}^{n}$ is a critical point of $\mathsf{F},$ and that $0\in\mathbb{R}^{1}$ is a critical value. The same definition of course works
if we replace $\mathbb{R}^{n}$ by a smooth manifold.

\begin{lemma}
Let $l_{1},...,l_{m}$ be linear functionals on $\mathbb{R}^{n}$. The two conditions are equivalent:

\vskip .1cm
{\rm\bf 1.} The function $\Delta\left(  x\right)  =\min_{i}l_{i}\left( x\right)  $ is non-positive on $\mathbb{R}^{n}$.

\vskip .1cm
{\rm\bf 2.} There is a convex linear relation between $l_{i},$ i.e. for some $\lambda_{1},...,\lambda_{r}>0$ and some $1\leq i_{1}<i_{2}<...<i_{r}\leq m$
\[ \lambda_{1}l_{i_{1}}+...+\lambda_{r}l_{i_{r}}=0\ .\]
\end{lemma}

\begin{proof}
{\rm\bf 2} $\Rightarrow$ {\rm\bf 1.} If $\lambda_{1}l_{i_{1}}+...+\lambda_{r}l_{i_{r}}=0$ then, evidently, for
every $x\in \mathbb{R}^{n}$ there is an index $i_{j}$ such that $l_{i_{j}}\left(  x\right) \leq 0.$

\vskip .2cm
{\rm\bf 1} $\Rightarrow$ {\rm\bf 2.} It is helpful to introduce an Euclidean structure on $\mathbb{R}^{n}$ with a scalar product $\left\langle ,\right\rangle$, so that
each functional $l_{i}\left( \ast\right)  $ can be written as $\left\langle v_{i},\ast\right\rangle$, with a nonzero vector $v_{i}$. Let $H_{i}\subset \mathbb{R}^{n}$ be
the halfspace defined by $H_{i}=\left\{ x\,  \vert\,\left\langle v_{i},x\right\rangle \leq0\right\}$. The condition $\Delta\left(  x\right)  \leq 0$ for all $x$ means that
$\cup_{i}H_{i}=\mathbb{R}^{n}$.

\vskip .2cm
Let $P$ be the convex envelope of the tips of the vectors $v_{1},...,v_{m}$. We claim that $\mathbf{0}\in P$ proving the implication. Indeed, in the opposite case
there is an affine hyperplane $N\subset M$ separating $\mathbf{0}$ from $P.$ Let $n$ be the normal to $N,$ pointing into the halfspace containing $P.$ The
scalar products $\left\langle v_{i},n\right\rangle$ are all positive which is a contradiction.\end{proof}

\vskip .2cm
All the critical clusters of balls which were considered in previous sections were critical points of the function $D$ in the above sense.

\vskip .2cm
Let $\bzero\in\mathbb{R}^{n}$ be a critical point of $\mathsf{F}.$ Define the subset $V^{0}\subset T_{\bzero}\mathbb{R}^{n}$ by $V^{0}:=\left\{  x\,\vert\,\Delta\left(
x\right)  =0\right\}$.

\begin{lemma} The set $V^{0}$ is convex.\end{lemma}

\begin{proof}For linear functionals $l_{1},\ldots,l_{m}$ on $\mathbb{R}^{n}$, let $\Delta\left(  x\right):  =\min_{i}l_{i}\left(  x\right)$ and
$V^{\geq}=\left\{  x\,\vert\,\Delta\left(  x\right)  \geq0\right\}$. The set $V^{\geq}=\cap_{i}\left\{  x\,\vert\, l_{i}\left(  x\right)  \geq0\right\}  $ is evidently convex. In our case
 the function $\Delta$ is non-positive, so $V^{0}=V^{\geq}\,$ is convex.\end{proof}

\vskip .2cm
Let $\bzero\in\mathbb{R}^{n}$ be a critical point of $\mathsf{F}$. Let $E=\cap_{i}\left\{  x\,\vert\, l_{i}\left(  x\right)  =0\right\}\subset V^0$. We claim that $E$
is the maximal linear subspace contained in the set $V^0$. Indeed suppose $y\in V^0\setminus E$. Then for some $j$ either $l_{j}\left(  y\right)<0$ or
 $l_{j}\left(  -y\right)<0$. Therefore the linear space $E\oplus\mathbb{R}y$ is not contained in $V^0$.

\vskip .2cm
The number $N\left(\bzero\right):=\dim E$ we call {\it null-index} of the critical point $\bzero$. For example, $N_{FCC}  =
N_{HCP} =1,$ $N_{C_{\mathfrak{m}}} =4,$ $N_{O_{6}}  =6$ -- if understood $\operatorname{mod}\left(  SO\left(  3\right)  \right)$.
For $m=1$ the space $E$ is the whole tangent space $T_{\bzero}\mathbb{R}^{n}$.

\vskip .2cm
We shall now establish a sufficient condition which ensures that the point $\mathbf{0}\in\mathbb{R}^{n}$ is a sharp local maximum
of the function ${\sf F}$, see (\ref{deffuF}).

\vskip .2cm
We assume that the family $\left\{  F_{1}\left(  x\right)  ,\dots,F_{m}\left( x\right)  \right\}  $ of $m$ analytic functions in $n$ variables, $m\leq n$,
possesses the following properties.
\begin{itemize}
\item[(A)]  The linear space, generated by the linear forms $l_{1},\dots,l_{m}$,
is $(m-k)$ dimensional, with $k$ positive. 
\item[(B)] The collection $\left\{  l_{1},\dots,l_{m}\right\}$ of linear forms can be split into $k$ subcollections $\left\{  l_{1},\dots,l_{m_{1}}\right\}  ,$
$\left\{  l_{m_{1}+1},\dots,l_{m_{2}}\right\} ,$ $\ldots ,$ $\left\{  l_{m_{k-1} +1},\dots,l_{m}\right\}  $
with non-intersecting spans, with \textbf{exactly one} linear relation between the functionals in each subcollection.
\item[(C)] For each $p=1,\dots,k$ the linear relation, from the property (B), between the functionals $\left\{  l_{m_{p-1}+1},\dots,l_{m_{p+1}}\right\}$
is strictly convex:
\begin{equation} \lambda_{p}^{1}l_{m_{p-1}+1}+\ldots+\lambda_{p}^{m_{p}}l_{m_{p}}=0\ ,\label{s2}\end{equation}
with $\lambda_{p}^{s}>0\ ,\ m_{p-1}+1\leq s\leq m_{p},$ $1\leq p\leq k$.
\item[(D)] For
\[ E_{p}=\bigcap_{u=m_{p-1}+1}^{m_{p}}\;\ker l_{u}\ ,\ E=\bigcap_{p=1}^{k}E_{p}\ ,\]
and $k$ quadratic forms $Q_{p},$ $1\leq p\leq k,$ defined by
\begin{equation} Q_{p}=\lambda_{p}^{1}q_{m_{p-1}+1}+\ldots+\lambda_{p}^{m_{p}}q_{m_{p}}\ ,\ \label{s21}\end{equation}
the inequality
\begin{equation} \min\left\{  Q_{1}(\xi),...,Q_{k}(\xi)\right\}  |_{_{\xi\in E}}\geq 0\label{s22} \end{equation}
admits only the trivial solution $\xi=0$.
\end{itemize}

\begin{theorem}
\label{T} (\cite{OS3}) Under the conditions \textrm{(A)} -- \textrm{(D)}, the origin $\bzero\in\mathbb{R}^{n}$ is a strict local maximum of the function
$\mathsf{F}(x)$.
\end{theorem}

In \cite{OS2} we were using a special case of this theorem, with $k=1,$ which is simpler. It then becomes an `if and only if' statement.

\vskip .2cm
\textbf{Note.} If $m=1$ we have the situation of a Morse function $F_1$. Indeed, $k$ must be equal to 1 by the property (A) and, by (B),
the linear functional $l_1$ vanishes.

\section{Proof of Theorem \ref{T}}

In this section we present a proof, having a more geometric flavor than the one given in \cite{OS3}, of Theorem \ref{T}. In the
first subsection we recall the proof, taken from \cite{OS2}, for the special case $k=1,$ since in this
case the notation is lighter. The general case is treated in the second subsection.

\subsection{Case $k=1$}
The key object of the proof is the set
\begin{equation} \mathcal{E}=\left\{  x\in\mathbb{R}^{n}:F_{1}\left(  x\right)  =\ldots =F_{m}\left(  x\right)  \right\}  \ . \label{91} \end{equation}

We assume that all occurring real vector spaces are equipped with a Euclidean structure. For a vector $v$ we denote by $\hat{v}$ the unit vector in the
direction of the vector $v$.

\vskip .2cm
Our proof will use the following observation.

\begin{lemma}
\label{lemobse} Let $\underline{\lambda}=\{\lambda^{1},\ldots,\lambda^{m}\}$ be a collection of $m$ positive real numbers, $\lambda^{j}>0$, $j=1,\dots,m$.
Let $\mathcal{W}_{\underline{\lambda}}$ be the space of $m$-tuples $\{ v_{1},\ldots,v_{m}\}$ of vectors in $\mathbb{R}^{m-1}$, generating the space
$\mathbb{R}^{m-1}$ and such that
\begin{equation} \lambda^{1}v_{1}+\ldots+\lambda^{m}v_{m}=0\ . \label{colire} \end{equation}
Then there exists a continuous positive-valued function $\delta\colon \mathcal{W}_{\underline{\lambda}}\to\mathbb{R}_{>0}$ such that for any unit
vector $\mathsf{s}\in\mathbb{R}^{m-1}$ we have
\begin{equation} \min_{i}\left\langle \mathsf{s},\hat{v}_{i}\right\rangle <-\delta\left( v_{1},\dots,v_{m} \right)  . \label{111} \end{equation}
\end{lemma}

\begin{proof}
For an angle $\alpha$, $0\leq\alpha<\pi$, let $D_{j}\left(  \alpha\right)$, $j=1,\dots,m$, denote the open spherical cap, centered at ($-\hat{v}_{j}$), on
the unit sphere $\mathbb{S}^{m-2}$, consisting of all the points $\mathsf{s}\in\mathbb{S}^{m-2}$ such that the angle $\measuredangle\left(
\mathsf{s},\hat{v}_{j}\right)  >\alpha$.

\vskip .2cm
For any unit vector $\mathsf{s}$ there exists an index $i$ such that $\left\langle \mathsf{s},v_{i}\right\rangle <0$. Indeed, since the
vectors $v_{1},\ldots,v_{m}$ span the whole space $\mathbb{R}^{m-1}$, some of the scalar products $\left\langle \mathsf{s},v_{j}\right\rangle $,
$j=1,\dots,m$, are nonzero. Taking the scalar product of the relation (\ref{colire}) with the vector $\mathsf{s}$ we see that at least one of the
scalar products $\left\langle \mathsf{s},v_{i}\right\rangle $ has to be negative. Therefore
\[ \bigcup_{i=1}^{m}\, D_{i}\left(  \frac{\pi}{2}\right)  =\mathbb{S}^{m-2}\ .\]
Thus,
\[ \alpha_{0}\left(  v_{1},\ldots,v_{m}\right)  >\frac{\pi}{2}\ ,\]
where the function $\alpha_{0}\left(  v_{1},\ldots,v_{m}\right)$ is defined by
\[ \alpha_{0}\left(  v_{1},\ldots,v_{m}\right)  =\sup\left\{  \alpha :\bigcup_{i=1}^{m}D_{i}\left(  \alpha\right)  =\mathbb{S}^{m-2}\right\}  \ .\]
Let
\[ \bar{\alpha}\left(  v_{1},\ldots,v_{m}\right)  :=\frac{1}{2}\left[  \alpha_{0}\left(  v_{1},\ldots,v_{m}\right)  +\frac{\pi}{2}\right]  \ .\]
Clearly, $\bigcup_{i=1}^{m} D_{i}\left(  \bar{\alpha}\right)  =\mathbb{S}^{m-2}$. Define the function $\delta$ by
\[ \delta\left(  v_{1},\ldots,v_{m}\right)  =-\cos\bar{\alpha}\left( v_{1},\ldots,v_{m}\right)  .\]
With this choice of the function $\delta$ the relation $\left(  \ref{111} \right)  $ clearly holds. The positivity and the continuity of the function
$\delta$ are straightforward.
\end{proof}

\vskip .2cm
We return to the consideration of our analytic functions.

\begin{lemma}
\label{prele} If the point $y\in\mathbb{R}^{n}$ happens to be away from the set $\mathcal{E}$, see (\ref{91}), and the norm $\left\Vert y\right\Vert $ is
small enough then one can find a point $x$ on $\mathcal{E}$ such that $\mathsf{F}\left(  y\right)  <\mathsf{F}\left(  x\right)  .$

\vskip .2cm
Moreover, there exists a constant $c>0$ such that for $y\notin \mathcal{E},$ and $x=x\left(  y\right)  \in\mathcal{E}$ being the point in
$\mathcal{E}$ closest to $y$ we have
\begin{equation} \mathsf{F}\left(  y\right)  <\mathsf{F}\left(  x\right)  -c\left\Vert x-y\right\Vert , \label{85}\end{equation}
provided, again, that the norm $\left\Vert y\right\Vert $ is small enough.
\end{lemma}

\begin{proof}
Since there is only one linear dependency between the differentials $l_{1},\dots,l_{m}$ of the functions $F_{1}(x),\dots,F_{m}(x)$, the set
$\mathcal{E}$ is a smooth manifold in a vicinity of the origin, of dimension $n-m+1$.

\vskip .2cm
We introduce the tubular neighborhood $U_{r}\left(  \mathcal{E} \right)  $ of the manifold $\mathcal{E}$, which is comprised by all points $y$
of $\mathbb{R}^{n}$ which can be represented as $\left(  x,\mathsf{s} _{x}\right)  ,$ where $x\in\mathcal{E}$ and $\mathsf{s}_{x}$ is a vector
normal to $\mathcal{E}$ at $x,$ with norm less than $r.$ Let $\mathcal{E}_{r^{\prime}}\subset\mathcal{E}$ be the neighborhood of the origin in
$\mathcal{E}$, comprised by all $x\in\mathcal{E}$ with norm $\left\Vert x\right\Vert <r^{\prime},$ and $U_{r}\left(  \mathcal{E}_{r^{\prime}}\right)
$ be the part of $U_{r}\left(  \mathcal{E}\right)  $ formed by points hanging over $\mathcal{E}_{r^{\prime}}.$ If both $r$ and $r^{\prime}$ are small enough
then every $y\in U_{r}\left(  \mathcal{E}_{r^{\prime}}\right)  $ can be written as $\left(  x,\mathsf{s}_{x}\right)  $ with $x\in\mathcal{E}
_{r^{\prime}}$ in a unique way. Note that $x$ is the point on $\mathcal{E}$ closest to $y$. Also, for any $r,r^{\prime}>0$ the set $U_{r}\left(
\mathcal{E}_{r^{\prime}}\right)  $ evidently contains an open neighborhood of the origin.

\vskip .2cm
Now we are going to show that if $y=\left(  x,\mathsf{s}_{x}\right)  \in U_{r}\left(  \mathcal{E}_{r^{\prime}}\right)  ,$
$\mathsf{s}_{x}\neq 0,$ and both $r$ and $r^{\prime}$ are small enough then $\mathsf{F}\left(  y\right)  <\mathsf{F}\left(  x\right)  $. To this end, let
$N_{x}$ be the plane normal to $\mathcal{E}$ at $x$ (so that $\mathsf{s}_{x}\in N_{x}$). We identify $N_{x}$ with the linear space $\mathbb{R}^{m-1},$
so that $x$ corresponds to $0\in\mathbb{R}^{m-1}$.

\vskip.2cm
Now we will use Lemma \ref{lemobse}, applied not to a single space, but to the whole collection of the $\left(  m-1\right)  $-dimensional spaces
$N_{x},$ $x\in\mathcal{E}_{r^{\prime}}.$ To do this, we equip each $N_{x}$ with $m$ vectors $v_{1}^{x},\ldots,v_{m}^{x}\in N_{x},$ which generate $N_{x}$
and which satisfy the same convex linear relation. All this data is readily supplied by the linear functionals $l_{1},\ldots,l_{m},$ restricted to
$N_{x}.$ Indeed, each restricted functional $l_{j}^{x}\equiv l_{j}{|}_{N_{x}}$ can be uniquely written as $l_{j}^{x}\left(  \ast\right)  =\left\langle
\ast,v_{j}^{x}\right\rangle ,$ with $v_{j}^{x}\in N_{x}.$ Here the scalar product on $N_{x}$ is the one restricted from $\mathbb{R}^{n}.$ Clearly, for
every $x$ we have
\[ \lambda^{1}v_{1}^{x}+\ldots+\lambda^{m}v_{m}^{x}=0\ ,\]
since for every vector $\mathsf{s}\in N_{x}$ we have $\lambda^{1}l_{1}\left( \mathsf{s}\right)  +\ldots+\lambda^{m}l_{m}\left(  \mathsf{s}\right)  =0$ (as
for any other vector). Moreover, $l_{j}\left(  \mathsf{s}\right)  <0$ for some $j=j(\mathsf{s})$, $1\leq j\leq m$, see  the proof of Lemma \ref{lemobse}.

\vskip .2cm
Since the space $N_{x=\bzero}$ is orthogonal to the null-space $E$, the $m$ vectors $v_{1}^{0},\ldots,v_{m}^{0}$ do generate $N_{\bzero}$. Because the
spaces $N_{x}$ depend on $x$ continuously, all of them are transversal to $E$, provided $r^{\prime}$ is small. Thus, the vectors $v_{1}^{x},\ldots,v_{m}^{x}$
do generate the spaces $N_{x}$ for all $x\in\mathcal{E}_{r^{\prime}},$ provided again that $r^{\prime}$ is small enough. Lemma \ref{lemobse} provides
us now with a collection of functions $\delta^{x}$ on the spaces $\mathcal{W}_{\underline{\lambda}}^{x}$ of $m$-tuples of vectors from $N_{x}.$
It follows from the continuity, in $x$, of the spaces $N_{x}$ and the $m$-tuples $\{v_{1}^{x},\ldots,v_{m}^{x}\}$, and from the Lemma \ref{lemobse}
that the functions $\delta^{x}$ can be chosen in such a way that the resulting positive function $\Delta(x):=\delta^{x}\left(  v_{1}^{x},\ldots,v_{m}
^{x}\right)  $ on $\mathcal{E}_{r^{\prime}}$ is continuous in $x$ and also is uniformly positive, that is,
\[ \Delta(x)>2c\text{ for all }x\in\mathcal{E}_{r^{\prime}}\ ,\]
for some $c>0$, provided $r^{\prime}$ is small enough.

\vskip .2cm
In virtue of Lemma \ref{lemobse}, for every $x\in\mathcal{E}_{r^{\prime}}$ and each vector $\mathsf{s}\in N_{x}$ there exists an index
$j\left(  \mathsf{s}\right)  $ for which the value $l_{j\left(  \mathsf{s} \right)  }\left(  \mathsf{s}\right)  $ of the functional $l_{j\left(
\mathsf{s}\right)  }$ is not only negative but, moreover, satisfies
\begin{equation} l_{j\left(  \mathsf{s}\right)  }\left(  \mathsf{s}\right)  <-2c\left\Vert \mathsf{s}\right\Vert .\label{86} \end{equation}
Hence for $y=\left(  x,\mathsf{s}_{x}\right)  \in U_{r}\left(  \mathcal{E}_{r^{\prime}}\right)  $ we have
\begin{equation} F_{j\left(  \mathsf{s}_{x}\right)  }\left(  y\right)  <F_{j\left( \mathsf{s}_{x}\right)  }\left(  x\right)  -c\left\Vert \mathsf{s}_{x}\right\Vert \label{186}
\end{equation}
provided both $r$ and $r^{\prime}$ are small. Therefore
\[ \min_{j}\left\{  F_{j}\left(  y\right)  \right\}  \leq F_{j\left( \mathsf{s}_{x}\right)  }\left(  y\right)  <F_{j\left(  \mathsf{s}_{x}\right)
}\left(  x\right)  -c\left\Vert \mathsf{s}_{x}\right\Vert =\min_{j}\left\{ F_{j}\left(  x\right)  \right\}  -c\left\Vert \mathsf{s}_{x}\right\Vert \ ,\]
where the last equality holds since $F_{1}\left(  x\right)  =\ldots =F_{m}\left(  x\right)  $, so we are done.
\end{proof}

\vskip .2cm
Theorem \ref{T} is a straightforward consequence of the next Proposition.

\begin{proposition}
\label{lq1} The point $x=\bzero$ is a sharp local maximum of the function $\mathsf{F}$ \emph{if} the form
\begin{equation} \sum_{u=1}^{m}\,\lambda^{u}q_{u} \label{10}\end{equation}
is negative definite on $E$.

\vskip .2cm
In the special case when all the functions $F_{u}(x)$, $u=1,\dots,m$, are linear-quadratic, i.e. $F_{u}$ are sums of linear and quadratic forms,
\begin{equation} F_{u}(x)=l_{uj}x^{j}+q_{ujk}x^{j} x^{k}\ , \label{linquadfun}\end{equation}
the \emph{if} statement becomes the \emph{iff} statement.
\end{proposition}

\begin{proof}
In view of Lemma \ref{prele} we can restrict our search of the maximum of the function $\mathsf{F}$ to the submanifold $\mathcal{E}$.

\vskip .2cm
Note that the plane $E$ is the tangent plane to $\mathcal{E}$ at the point $\bzero\in\mathcal{E}$, so the coordinate projection of $\mathcal{E}$ to
$E$ introduces the local coordinates on $\mathcal{E}$ in a vicinity of $\bzero$. As a result, $\mathcal{E}$ can be viewed as a graph of a function $Z$ on $E$,
$Z\left(  \mathbf{x}\right)  \in\mathbb{R}^{m-1}$ $:$
\[ \mathcal{E}=\left\{  \mathbf{x,z}:\mathbf{x}\in E,\mathbf{z}=\left( z_{1}\left(  \mathbf{x}\right)  ,\ldots,z_{m-1}\left(  \mathbf{x}\right) \right)  \right\}  .\]
This is an instance of the implicit function theorem. The point $\mathbf{x}=\bzero$ is a critical point of all the functions $z_{l}\left(  \mathbf{x}\right)  .$

\vskip .2cm
Denote by $M$ the restriction of any of the functions $F_{i}$ to $\mathcal{E}$. Clearly, it is a smooth function, and the differential $dM$
vanishes at $\bzero\in\mathcal{E}$. So our proposition would follow once we check that\textbf{ }the second quadratic form of $M$ at $\bzero$ is twice the
form\textbf{ }$\left(  \ref{10}\right)  .$ To see that, let us compute the derivative $\frac{d^{2}M}{dx_{1}^{2}}$ at the origin; the computation of other
second derivatives repeats this computation. We have
\begin{align*}
&  \frac{d}{dx_{1}}M\left(  \mathbf{x,z}\left(  \mathbf{x}\right)  \right) =\left(  \!\frac{\partial}{\partial x_{1}}M\!\right)  \!\left(  \mathbf{x,z}
\left(  \mathbf{x}\right)  \right)  \\
&  \!+\left(  \!\frac{\partial}{\partial x_{n-m+2}}M\!\right)  \!\left( \mathbf{x,z}\left(  \mathbf{x}\right)  \right)  \frac{\partial}{\partial
x_{1}}z_{1}\left(  \mathbf{x}\right)  +\ldots+\left(  \!\frac{\partial }{\partial x_{n}}M\!\right)  \!\left(  \mathbf{x,z}\left(  \mathbf{x}\right)
\right)  \frac{\partial}{\partial x_{1}}z_{m-1}\left(  \mathbf{x}\right)  ,
\end{align*}
and then
\begin{align*}
\frac{d^{2}}{dx_{1}^{2}}M\left(  \mathbf{x,z}\left(  \mathbf{x}\right) \right)  {|}_{\mathbf{x}=\bzero} &  =2\left[  q_{1}\right]  _{1,1}+\left[
l_{1}\right]  _{1}\cdot0\text{ (since all }\frac{\partial}{\partial x_{1} }z_{l}\left(  \mathbf{0}\right)  =0\text{)}\\
&  +\left[  l_{1}\right]  _{n-m+2}\cdot\frac{\partial^{2}}{\partial x_{1}^{2} }z_{1}\left(  \mathbf{0}\right)  +\ldots+\left[  l_{1}\right]  _{n}\cdot
\frac{\partial^{2}}{\partial x_{1}^{2}}z_{m-1}\left(  \mathbf{0}\right)  .
\end{align*}
Let us introduce the vector
\[ \Delta=\left(  0,\ldots,\frac{\partial^{2}}{\partial x_{1}^{2}}z_{1}\left( \mathbf{0}\right)  ,\ldots,\frac{\partial^{2}}{\partial x_{1}^{2}}
z_{m-1}\left(  \mathbf{0}\right)  \right) \  .\]
Then we have
\[ \frac{d^{2}}{dx_{1}^{2}}M_{1}\left(  \mathbf{x,z}\left(  \mathbf{x}\right) \right)  {|}_{\mathbf{x}=\bzero}=2\left[  q_{1}\right]  _{1,1}+l_{1}\left(
\Delta\right) \  .\]
Since we have $m-1$ identities
\[ M_{1}\left(  \mathbf{x,z}\left(  \mathbf{x}\right)  \right)  =M_{2}\left( \mathbf{x,z}\left(  \mathbf{x}\right)  \right)  =M_{m}\left(  \mathbf{x,z}
\left(  \mathbf{x}\right)  \right) \  ,\]
we can write also
\[ \frac{d^{2}}{dx_{1}^{2}}M\left(  \mathbf{x,z}\left(  \mathbf{x}\right) \right)  {|}_{\mathbf{x}=\bzero}=2\left[  q_{l}\right]  _{1,1}+l_{l}\left(
\Delta\right)  ,\ l=2,\ldots,m\ \ .\]
By $\left(  \ref{s2}\right)  $ we then have
\[ \frac{d^{2}}{dx_{1}^{2}}M\left(  \mathbf{x,z}\left(  \mathbf{x}\right)
\right)  {|}_{\mathbf{x}=\bzero}=2\left(  \sum_{l}\lambda^{l}\left[  q_{l}\right]
_{1,1}\right)  ,\]
so our claim follows. \end{proof}

\subsection{General case}
Let us introduce the functions
\[ \mathsf{F}_{p}\left(  x\right)  :=\min\left\{  F_{m_{p-1}+1}\left(  x\right)
,\dots,F_{m_{p}}\left(  x\right)  \right\}  \ \]
and the manifolds
\[ \mathcal{E}_{p}=\left\{  x\in\mathbb{R}^{n}:F_{m_{p-1}+1}\left(  x\right) =\dots=F_{m_{p}}\left(  x\right)  \right\}  \ .\]
In the vicinity of the origin $\mathbf{0}\in\mathbb{R}^{n}$ the manifolds $\mathcal{E}_{p}$ meet in general position, due to the conditions (A), (B), so
their intersection
\[ \mathcal{E=}\bigcap_{p=1}^{k}\,\mathcal{E}_{p}\ ,\]
is a smooth manifold as well, of dimension $n-m+k$. As we know from the previous section, the point $\mathbf{0}\in\mathcal{E}_{p}$ is a critical point
of the restriction of the function $\mathsf{F}_{p}$ to $\mathcal{E}_{p}$, and its second differential equals to the form $Q_{p},$ $p=1,...,k$. Hence it
follows from (C) that the function $\mathsf{F}\left(  x\right)  =\min\left\{ \mathsf{F}_{1}\left(  x\right)  ,\dots,\mathsf{F}_{k}\left(  x\right)
\right\}  ,$ restricted to $\mathcal{E},$ is negative in the vicinity $W\subset\mathcal{E}$ of the point $x=\mathbf{0,}$ except at the point
$\mathbf{0,}$ where $\mathsf{F}\left(  \mathbf{0}\right)  =0$ (see the end of the first proof in the previous section).

\vskip.2cm As for $k=1$, it would be nice to show that if a point $y\in\mathbb{R}^{n}$ happens to be away from $\mathcal{E},$ while $\left\Vert
y\right\Vert $ is small enough, then one can find a point $x$ in $W$ such that for all $z\in\left[  x,y\right]  $ we have
\begin{equation} F_{u}\left(  z\right)  <F_{u}\left(  x\right)  \mathbf{<0}\label{50}\end{equation}
for some $u=1,...,m.$ Then we would be done. It seems, however, that it is not necessarily the case. We will establish a weaker property, which is also
sufficient for our purposes. Let $\mathcal{E}(r^{\prime})\subset\mathcal{E}$ be a ball, centered at the origin, of radius $r^{\prime}$ in $\mathcal{E}$,
and let $U_{r}\left(  \mathcal{E}(r^{\prime})\right)  $ be a tubular neighborhood of $\mathcal{E}(r^{\prime})$, with both $r$ and $r^{\prime}$
being small enough. Let us represent $U_{r}\left(  \mathcal{E}_{r^{\prime} }\right)  $ as a union of segments of the form $[\left(  x,0\right)  ,\left(
x,\mathsf{s}_{x}\right)  [\,$, which do not intersect outside $\mathcal{E},$ where each vector $\mathsf{s}_{x}$ is normal to $\mathcal{E}$ at $x.$ We will
show that for each $x,\mathsf{s}_{x}$ and $t\in\left(  0,1\right)  $ one can find and index $u=u\left(  x,\mathsf{s}_{x},t\right)  $ such that
$F_{u}\left(  x,t\mathsf{s}_{x}\right)  <0.$

\vskip.2cm The normal vector $\mathsf{s}_{x}$ is an element of the normal vector space $N_{x}$. $N_{x}$ can be decomposed into direct sum of $k$ vector
spaces, $N_{x}=\oplus_{p=1}^{k}N_{x}^{p},$ where each space $N_{x}^{p}$ is generated by the gradients, at $x$, of the functions $F_{m_{p-1}+1}\left(
x\right)  ,\dots,F_{m_{p}}\left(  x\right)  $. Without loss of generality we can suppose that the subspaces $N_{x}^{p}$ are orthogonal, by changing the
Euclidean structure. Due to $\left(  \ref{s22}\right)  ,$ we can suppose the existence of a value $\mathfrak{v}>0$ such that
\[ \min\left\{  Q_{1}(\xi),...,Q_{k}(\xi)\right\}  |_{_{\xi\in E}}\leq -\mathfrak{v}\left\Vert \xi\right\Vert ^{2}\ .\]
Therefore, given $x\in\mathcal{E}$ with $\left\Vert x\right\Vert =\varepsilon$, we can suppose without loss of generality that
\begin{equation} \mathsf{F}_{1}\left(  \left(  x,0\right)  \right)  \leq-\frac{\mathfrak{v}} {2}\,\varepsilon^{2}\ .\label{s3}\end{equation}

\vskip .2cm
We will proceed in five steps.

\vskip .2cm
\textbf{1. }In the easy case when our vector $\mathsf{s}_{x}\in N_{x}$ is a vector from the subspace $N_{x}^{1},$ we are done, as in the
previous section, since we know from relation $\left(  \ref{186}\right)$ that for some $h\in\left\{  1,...,m_{1}\right\}  $ the function $F_{h}$
satisfies $\left(  \ref{186}\right)$, and so
\begin{equation} \mathsf{F}_{1}\left(  \left(  x,\mathsf{s}_{x}\right)  \right)  <\mathsf{F}_{1}\left(  \left(  x,0\right)  \right)  -c_{1}\left\Vert \mathsf{s}
_{x}\right\Vert \ ,\label{s1}\end{equation}
where the value $c_{1}>0$ is determined by the functionals $l_{1} ,\dots,l_{m_{1}}.$

\vskip .2cm
\textbf{2. }Consider a more general case, when the vector $\mathsf{s}_{x}$ is not in $N_{x}^{1},$ but its first coordinate
$\mathsf{s}_{x}^{1}$ in the decomposition
\[ \mathsf{s}_{x}=\sum_{p}\mathsf{s}_{x}^{p}\ ,\ \mathsf{s}_{x}^{p}\in N_{x}^{p}\ ,\]
satisfies the relation: $\left\Vert \mathsf{s}_{x}^{1}\right\Vert \geq {\sf r}\left\Vert \mathsf{s}_{x}\right\Vert $ with some ${\sf r}>0.$ Let us denote the set
of all such vectors $\mathsf{s}_{x}$ by $C_{{\sf r}}\left(  N_{x}^{1}\right) \subset N_{x}$, this is the ${\sf r}$-cone around $N_{x}^{1}.$ The smaller the
constant ${\sf r}$ is, the bigger is the cone $C_{{\sf r}}\left(  N_{x}^{1}\right)$. For $\mathsf{s}_{x}\in C_{{\sf r}}\left(  N_{x}^{1}\right)  $ the relation $\left(
\ref{s1}\right)  $ still holds, but with $c_{1}$ replaced by ${\sf r}c_{1}.$

\vskip .2cm
\textbf{3a. }The delicate case therefore is when $\mathsf{s}_{x}$ is in $N_{x}^{2},$ say, because we do not have the relation $\left(
\ref{s1}\right)  $ anymore. The only thing we know is that among the differentials $d_{x}F_{1},\dots,d_{x}F_{m_{1}},$ \textit{computed at the
point} $x$ (which is indicated by the subscript in $d_{x}$), there is at least one, $d_{x}F_{h},$ $1\leq h\leq m_{1},$ for which
\begin{equation} d_{x}F_{h}\left(  \mathsf{s}_{x}^{p}\right)  \leq0\label{s4} \end{equation}
-- that follows from the condition (A); hence the estimate $\left( \ref{s1}\right)  $ might not hold. For all we know the function $F_{h}$ might
even grow along the segment $\left[  \left(  x,0\right)  ,\left( x,\mathsf{s}_{x}\right)  \right]  ,$ since we have no information about the
forms $q_{1},...,q_{m_{1}}$ outside $E.$ But we stress that the possible increase of the function $F_{h}$ is not linear, due to $\left(  \ref{s4}
\right)  $, so it is at least of the second order in $t.$ Hence, the function $F_{h}$ at the points $\left(  x,t\mathsf{s}_{x}\right)  $ is still negative,
provided $\left\Vert t\mathsf{s}_{x}\right\Vert <\frac{1}{C}\varepsilon$, once $C$ is big enough, because of the above mentioned at least quadratic in $t$
behavior of the function $F_{h}$ along the direction $\mathsf{s}_{x},$ and due also to $\left(  \ref{s3}\right)  .$

\vskip .2cm
\textbf{3b.} In order to treat the remaining part of the segment $\left[  \left(  x,0\right)  ,\left(  x,\mathsf{s}_{x}\right)  \right]  $ we
will use the functions $F_{m_{1}+1},\dots,F_{m_{2}}$. Note that at the point $x$ the quadratic forms $q_{m_{1}+1},\dots,q_{m_{2}}$ -- even being positive
-- do not get above the level $\left(  \tilde{C}\varepsilon\right)  ^{2},$ while (at least) one of the differentials $d_{x}F_{m_{1}+1},\dots
,d_{x}F_{m_{2}}$ -- say, $d_{x}F_{h^{\prime}}$ -- decays linearly along the direction of $\mathsf{s}_{x}:$
\[ d_{x}F_{h^{\prime}}\left(  t\mathsf{s}_{x}\right)  <-2c_{2}\left\Vert t\mathsf{s}_{x}\right\Vert \ ,\]
where the value $c_{2}$ is determined by the linear functionals $l_{m_{1}+1},\dots,l_{m_{2}}$, compare with $\left(  \ref{s1}\right)  $. In particular,
for $\left\Vert t\mathsf{s}_{x}\right\Vert \geq\frac{1}{C}\varepsilon$ we have $d_{x}F_{h^{\prime}}\left(  t\mathsf{s}_{x}\right)  <-2\frac{c_{2}}
{C}\varepsilon,$ which beats $\left(  \tilde{C}\varepsilon\right)  ^{2}$ once $\varepsilon$ is small enough. Therefore the function $F_{h^{\prime}}$ is
negative on the segment $\left[  \left(  x,t\mathsf{s}_{x}\right)  ,\left( x,\mathsf{s}_{x}\right)  \right]  $ once $\left\Vert t\mathsf{s}
_{x}\right\Vert \geq\frac{1}{C}\varepsilon,$ provided $\left(  x,\mathsf{s}_{x}\right)  \in U_{r}\left(  \mathcal{E}(r^{\prime})\right)  $ with $r,r^{\prime}$ small.

\vskip .2cm
\textbf{4. }The same argument applies to the case when $\mathsf{s}_{x}$ is not in $N_{x}^{2},$ but belongs to the cone $C_{{\sf r}}\left(
N_{x}^{2}\right)  $, see step \textbf{2} above.

\vskip.2cm \textbf{5. }Since the union of the cones coincides with $N_{x}$,
\[ \bigcup_{p=1}^{k}\,C_{{\sf r}}\left(  N_{x}^{p}\right)  =N_{x}\ ,\]
provided ${\sf r}$ is small enough, the proof is over. \myblacksquare

\section{Platonic clusters}
\subsection{$\delta$-rotation process}\label{delrotp}
The $\delta$-rotation process was introduced in \cite{OS4}. Here is its description, for the case of the cluster $O_{6}.$ Consider  the cluster
of the six tangent lines to the unit sphere, which contain the edges of the regular tetrahedron. The points of the sphere at which tangent lines pass are
the edge middles of the regular tetrahedron. The initial position of the edges of the tetrahedron in our $\delta$-rotation process are shown in blue on
Figure \ref{A4configurations}.

\begin{figure}[H]
\centering
\includegraphics[scale=0.7]{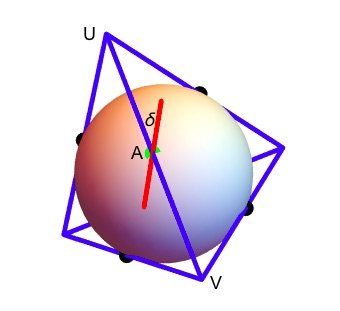} \vskip -.2cm\caption{Sphere tangent to tetrahedron edges}
\label{A4configurations}
\end{figure}

\vskip -.2cm
Then each edge is rotated around the diameter of the unit sphere, passing through the middle of the edge, by an angle $\delta$. On Figure
\ref{A4configurations} the point $A$ (in green) is the middle of the edge $UV$. The line, passing through the point $A$ and rotated by the angle
$\delta$, is shown in red. The lines passing through other middles of edges are rotated according to $\mathbb{A}_{4}$ symmetry. This is our $\delta
$-process. For $\delta=\pi/4$ this is exactly the cluster $O_{6}$.

\vskip .2cm
The distance function $d$ becomes the function of $\delta,$ see  Figure \ref{DistanceGraphA4}.

\begin{figure}[H]
\vspace{.2cm} \centering
\includegraphics[scale=0.6]{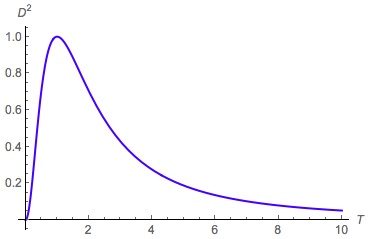} \vskip -.2cm\caption{Graph of $d^{2}(T)$, $T=\tan\delta$}
\label{DistanceGraphA4}
\end{figure}
It gets its maximal value at  $\delta=\pi/4,$ i.e. at the cluster $O_{6}$.

\vskip .2cm
A similar construction can be performed for each pair of dual Platonic bodies. Namely, let a unit sphere touch the edge middles of a Platonic body
$\mathcal{P}$. We can rotate all the edges of $\mathcal{P}$ around the axes passing through the
center of the sphere and tangency points by the angle $\delta$.  When $\delta$ reaches the value $\pi/2,$ the edges form the Platonic body dual to $\mathcal{P}$.

\vskip .2cm
For the pair octahedron-cube (respectively, icosahedron-dodecahedron) the function $d\left(  \delta\right) $ is shown on Figure
\ref{MinDistOC} (respectively, Figure \ref{PlotMinID}).

\begin{figure}[H]
\vspace{.2cm} \centering
$\!\!\!\!\!\!\!$\includegraphics[scale=0.38]{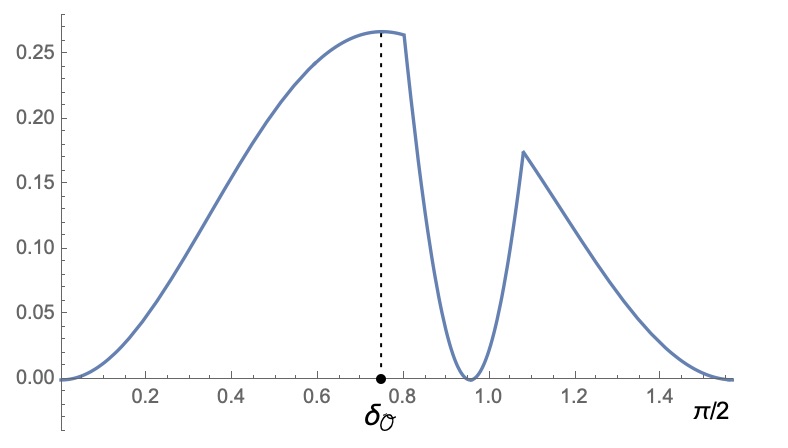}
\vskip -.2cm
\caption{Graph of the square of the minimal distance for the pair octahedron-cube}
\label{MinDistOC}
\end{figure}

\begin{figure}[H]
\centering
\includegraphics[scale=0.3]{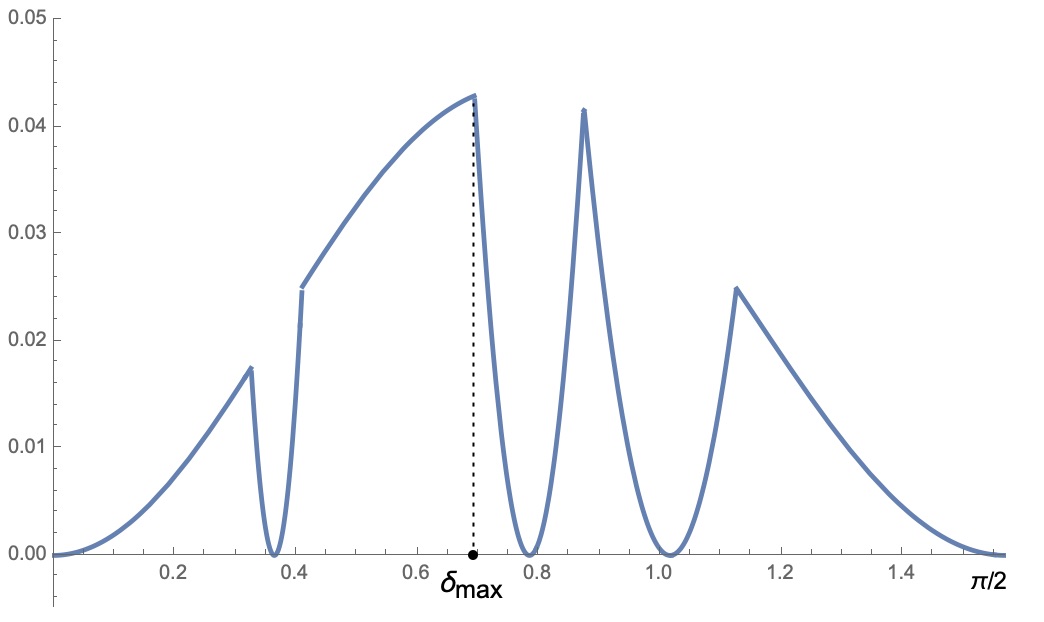} \vskip .2cm\caption{Graph of the
square of the minimal distance for the pair icosahedron-dodecahedron}
\label{PlotMinID}
\end{figure}

The resulting maximal clusters (of twelve cylinders, at the angle $\delta_{\mathcal{O}}$, for the pair octahedron-cube, and of thirty cylinders, at the angle
$\delta_{\text{max}}$, for the pair icosahedron-dodecahedron) are shown on Figures \ref{OC-1} and \ref{MaxID-ConfigViewFrom5-axis} respectively.

\begin{figure}[H]
\vspace{-.2cm}
\centering
\includegraphics[scale=0.38]{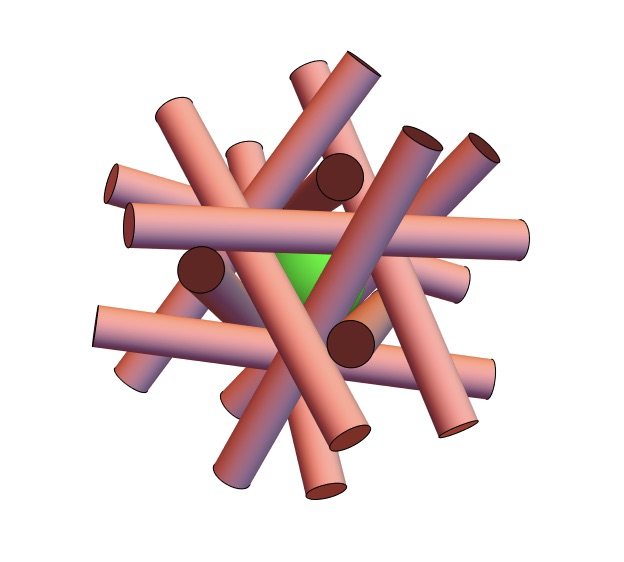}
\caption{Octahedron/cube maximal configuration of cylinders, view from a vertex of the cube}
\label{OC-1}
\end{figure}

\vspace{3cm}
\begin{figure}[H]
\centering
\includegraphics[scale=0.5]{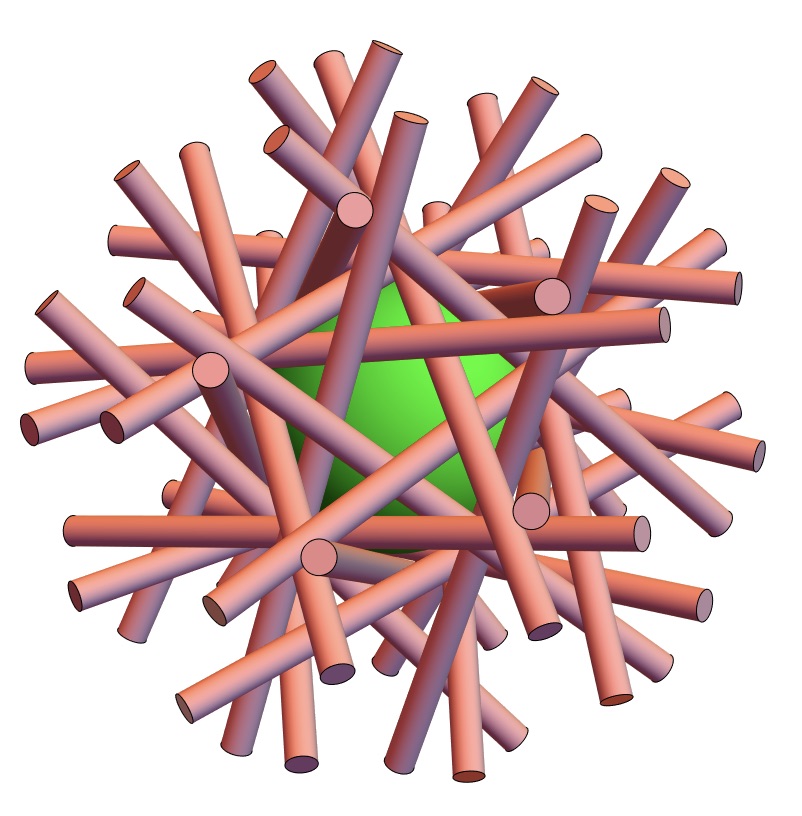}\caption{Maximal cluster, view from the tip of a 5-fold axis}
\label{MaxID-ConfigViewFrom5-axis}
\end{figure}

\subsection{Minimal clusters of tangent lines}
The clusters where the function $d$ vanishes are also quite interesting.

\vskip .2cm
For the pair octahedron-cube the minimum happens at the angle $\delta=\arctan \left(\sqrt{2}\right)$. The resulting figure, formed by four triangles of edge length $2\sqrt{3}$, is shown on Figure \ref{Octahedron-CubeMinimum}.
\begin{figure}[H]
\vspace{-.4cm}
\centering
\includegraphics[scale=0.32]{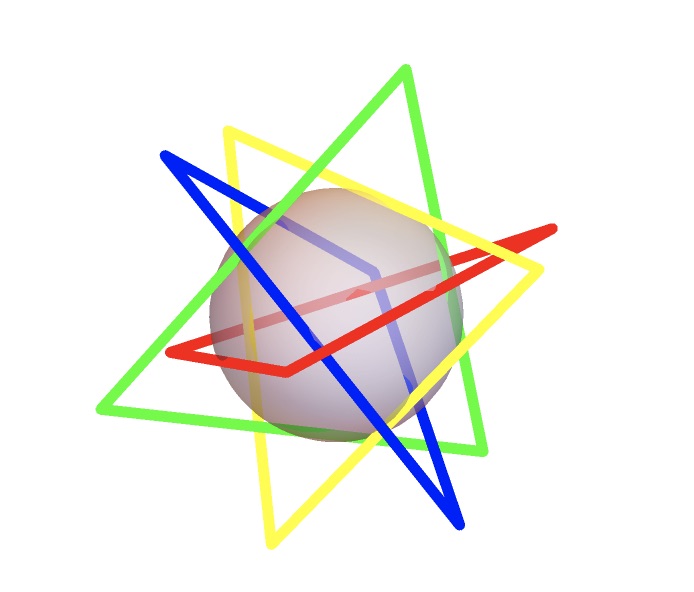}
\vskip -.2cm
\caption{Octahedron/cube minimum}
\label{Octahedron-CubeMinimum}
\end{figure}

For the pair icosahedron-dodecahedron there are several minima, see Graph \ref{PlotMinID}.  For example, the second minimum happens at $\delta=\frac{\pi}{4}$.
The thirty edges split into five one-skeletons of the tetrahedron. Thus we get the cluster of one-skeletons  of the five tetrahedra of edge length $2\sqrt{2}$,
inscribed in the dodecahedron. It is shown on Figure \ref{SecondMinimum}.
\begin{figure}[H]
\vspace{-.2cm}
\centering
\includegraphics[scale=0.54]{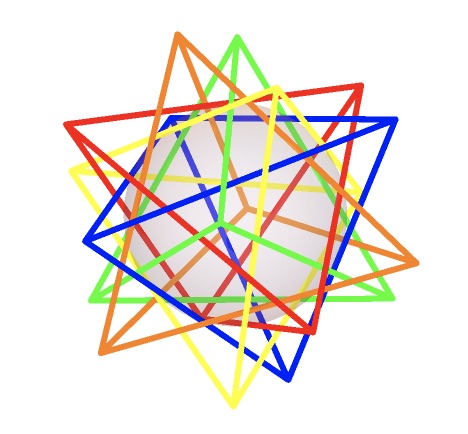}
\vskip -.2cm
\caption{Second minimum}
\label{SecondMinimum}
\end{figure}

\section{Conjectures and questions}\label{conjques}
In this last section we formulate some open problems.

\vskip .2cm
Let $\mathfrak{C}$ be a cluster of solid bodies $\Gamma_{1},\dots,\Gamma_{L}$, touching the unit sphere $\mathbb{S}^{n-1}\subset\mathbb{R}^{n}$. Then the
solid bodies $\Gamma_{j}\times\mathbb{R}$, $j=1,\dots,L$, touch the unit sphere $\mathbb{S}^{n}\subset\mathbb{R}^{n+1}$. We denote the so defined
cluster by $\mathrm{K}(\mathfrak{C})$; this construction is due to Kuperberg, hence our notation. For example, if $\mathfrak{C}$ is a cluster of six unit
discs touching the central unit disc, then $\mathrm{K}(\mathfrak{C})$ is the cluster $C_{6}$ of unit cylinders touching the unit sphere $\mathbb{S}
^{2}\subset\mathbb{R}^{3}$.

\vskip .2cm
Let $D$ be the distance function; $D(\mathfrak{C})$ is the value of the distance function on $\mathfrak{C}$. In case the cluster $\mathrm{K}
(\mathfrak{C})$ can be unlocked to (maybe several) locally maximal clusters in $\mathbb{R}^{n+1}$, choose the cluster with maximal value of distance function
among them and denote this value by $D_{1}(\mathfrak{C})$. Otherwise, put $D_{1}(\mathfrak{C})=D(\mathfrak{C}).$ Similarly, denote by $D_{j}
(\mathfrak{C})$ the corresponding value of the distance function for the cluster $\mathrm{K}^{j}(\mathfrak{C})$ in $\mathbb{R}^{n+j}$, $j=1,2,\dots$

\vskip .2cm
\textbf{q1}. Let $\mathfrak{C}$ in $\mathbb{R}^{n}$. We believe that under some conditions on $D(\mathfrak{C})$ the cluster $\mathrm{K}(\mathfrak{C})$ can
be unlocked. Plausibly, there exists a function $\mathfrak{d}(n)$ such that if $D(\mathfrak{C})$ is bounded from above by $\mathfrak{d}(n)$ then
$\mathrm{K}(\mathfrak{C})$ is unlockable. Note that a restriction on the value of $D(\mathfrak{C})$ is needed. For example, if $\mathfrak{C}$ is the maximal  cluster of two, or even three, congruent circles in $\mathbb{R}^{2}$ then $\mathrm{K}(\mathfrak{C})$ is rigid.

\vskip .2cm
\textbf{q2. } Let $\mathfrak{bd}(n)$ be the maximal possible value of the function $\mathfrak{d}(n)$, from \textbf{q1}, for clusters of a certain class,
say clusters of a congruent balls. We believe that {\bf(i)} the function $\mathfrak{bd}(n)$ does not decrease in $n$; {\bf(ii)} given a cluster $\mathfrak{C}$,
the function $D_{j}(\mathfrak{C})$ does not decrease in $j$. What is upper limit of $D_{j}(\mathfrak{C})$ when $j\rightarrow\infty$?

\vskip .2cm
\textbf{q3. }Twelve unit spheres $\mathbb{S}^{2}$ can touch the central unit sphere in $\mathbb{R}^{3}$. Motivated by (\textbf{q1}) and  (\textbf{q2}) we believe
that more and more space opens when we iterate the operation $\mathrm{K}$. Therefore the following question arises: what is the minimal $j$ such that thirteen
bodies $\mathbb{S}^{2}\times\mathbb{R}^{j}$ can touch the central unit sphere $\mathbb{S}^{2+j}$? Plausibly, $j=1$.

\vskip .2cm
\textbf{q4. }What are possible generalizations of the cluster $O_{6}$ to higher dimensions? Here a cylinder can be replaced by $\mathrm{C}_{a,b}
:=\mathbb{S}^{a}\times\mathbb{R}^{b}$ in $\mathbb{R}^{a+b+1}$. A more precise question: for which $a$ and $b$ are there obvious clusters of bodies congruent
to $\mathrm{C}_{a,b}$, having the distance function equal to 1? Can such clusters be obtained by some higher dimensional generalizations of the
$\delta$-rotation process (see Section \ref{delrotp}) applied to faces of certain dimension of higher simplices/octahedra/cubes?

\vskip .2cm
\textbf{q5. }Let $\mathcal{I}_{4}$ and $\mathcal{D}_{4}$ be four-dimensional analogues of the exceptional platonic bodies in $\mathbb{R}^{3}$. Is there a
version of the $\delta$-rotation process which produces the analogue of five tetrahedra inscribed into a dodecahedron?

\vskip .6cm\noindent{\textbf{Acknowledgements.} Part of the work of S. S. has been carried out in the framework of the Labex Archimede (ANR-11-LABX-0033)
and of the A*MIDEX project (ANR-11- IDEX-0001-02), funded by the Investissements d'Avenir French Government program managed by the French
National Research Agency (ANR). Part of the work of S. S. has been carried out at IITP RAS. The support of Russian Foundation for Sciences (project No.
14-50-00150) is gratefully acknowledged by S. S. The work of O. O. was supported by the Program of Competitive Growth of Kazan Federal University and
by the grant RFBR 17-01-00585.}

\end{document}